\RequirePackage{fix-cm}
\documentclass[matprg,nospthms]{svjour3} 
\smartqed  
\usepackage{graphicx}
\usepackage{microtype}
\usepackage{graphicx}%
\usepackage{multirow}%
\usepackage{amsmath,amssymb,amsfonts}%
\usepackage{amsthm}%
\usepackage{mathrsfs}%
\usepackage[title]{appendix}%
\usepackage{xcolor}%
\usepackage{textcomp}%
\usepackage{manyfoot}%
\usepackage{booktabs}%
\usepackage{algorithm}%
\usepackage{algorithmicx}%
\usepackage{algpseudocode}%
\usepackage{listings}%
\newtheorem{defn}{Definition}[section]

\newtheorem{thm}{Theorem}[section]
\newtheorem{prop}{Proposition}[section]

\newtheorem{example}{Example}[section]
\newtheorem{remark}{Remark}[section]

\newcommand{\R}{\mathbb{R}}

\newcommand{\ds}{\displaystyle}
\usepackage{amsmath}
\usepackage{hyperref}
\usepackage[misc]{ifsym}
\begin{document}

\title{Directional first order approach for a class of bilevel programs
}


\author{Kuang Bai  \and Wei Yao \and Jane J. Ye \and Jin Zhang }


\institute{K. Bai \at School of Mathematics, Hunan University, Hunan, China\\
              \email{baik@hnu.edu.cn}           
           \and
           W. Yao \at School of Artificial Intelligence, Southern University of Science and Technology, Guangdong, China\\ \email{yaow@sustech.edu.cn}
           \and
           J. J. Ye \at Department of Mathematics and Statistics, University of Victoria, B.C., Canada\\ \email{janeye@uvic.ca}
           \and
           J. Zhang \at Corresponding author. Department of Mathematics, Southern University of Science and Technology, Guangdong, China\\ \email{zhangj9@sustech.edu.cn}         }

\date{Received: date / Accepted: date}

\maketitle

\begin{abstract}
In this paper, we study a class of bilevel optimization program, where the feasible set of the lower level program is independent of the upper level variable. For bilevel programs it is known that the first order reformulation of a bilevel program requires the convexity of the lower level program while reformulations involving the value function result in  difficult optimization problems. In this paper we propose a directional first order approach which does not require  convexity of the lower level program. First, we propose  some conditions under which  the lower level program can be equivalently characterized by its first order condition over a directional neighborhood around the local optimal condition. Next we give some conditions under which the classical first order optimality condition in the form of M-stationary condition still holds as a necessary optimality condition for the first order  reformulation of the bilevel program  even when the lower level program is nonconvex.
\end{abstract}

\section{Introduction}
In this paper we consider the following bilevel program
\begin{equation}\label{BP}\tag{BP}
	\begin{aligned}
		\min_{x,y}\quad &F(x,y)\\
		\mathrm{s.t.}\quad &y\in S(x),\quad G(x,y)\leq 0,
	\end{aligned}
\end{equation}
where $S(x)$ denotes the solution set of the lower level program:
\begin{equation}\label{P(x)}\tag{P(x)}
	\min_{y}\  f(x,y) \quad \mbox{s.t.}\quad y\in Y,  
\end{equation}
where $Y\subseteq \mathbb R^m$ is a closed set, $F:\mathbb R^n\times\mathbb R^m\rightarrow\mathbb R,\ G:\mathbb R^n\times\mathbb R^m\rightarrow\mathbb R^q$ are continuously differentiable, and $f:\mathbb R^n\times\mathbb R^m\rightarrow\mathbb R$ is twice continuously differentiable. 
Program (\ref{BP}) is a special case of a general bilevel program in which the constraint set $Y$ is replaced by $Y(x)$, a set depending on the upper level variable $x$. 

Since the lower level constraint is independent of $x$, this kind of bilevel program is sometimes referred to as a simple bilevel program, see e.g.,  \cite{LMY14}. However it is not ``simple'' at all as the geometric and computational hardness of this class of problems have been  shown in the recent work~\cite{BLPV26a}. We should also point out that sometimes a simple bilevel program means that  all functions at both the upper and lower levels are independent of $x$ which is a special case of the bilevel program we study in this paper; see e.g., ~\cite{MS2023,MST2024,SS2017}. 

Throughout the paper we assume the solution set $S(x)$ is nonempty for each $x$.

Bilevel optimization was first proposed  by H. Von Stackelberg~\cite{Stackelberg} in 1934 to study problems in market economies where the decisions are made by two distinct decision-makers at separate hierarchical levels. Since then it has been widely applied in fields such as meta-learning and hyperparameter selection in machine learning \cite{GYYZZ,ML,tax,YYZZ}, principal-agent problems in economics \cite{M76,M99}; see \cite{Dempe,Luo} for more applications.  

Although (\ref{BP}) is simpler than the general bilevel program, it appears in many important applications; see \cite{Outrata,Renner15,Rogerson85} and the references within.


\subsection{Application in the principal-agent problem}

In the standard moral hazard model, a principal designs a contract to hire an agent to perform a project that yields a stochastic output. Let the output space be a finite set $S = \{s_1, \dots, s_n\}$ with cardinality $n$. The agent chooses an action $y \in Y$, and the resulting output $s \in S$ is distributed according to a probability function $P(s, y)$, where $P(s_j, y)$ represents the probability of output $s_j$ given action $y$. The principal cannot directly observe the agent's action but can observe the output, and thus designs a contract $x(s)$, which specifies the wage paid to the agent based on the observed output. 
The contract can be represented as a vector $x = (x(s_1), \dots, x(s_n))^T \in \mathbb{R}^n$, where $x_j:=x(s_j)$ is the wage corresponding to output $s_j$. 
The value of the output to the principal is given by the function $\pi(s) := (\pi(s_1), \dots, \pi(s_n))$, where $\pi_j:=\pi(s_j)$ denotes the value associated with output $s_j$. The principal's payoff for output $s_j$ is $v(\pi_j - x_j)$, 
and the agent’s payoff is $u(x_j) - c(y)$, where $v(\cdot)$ and $u(\cdot)$ are utility functions of the principal and agent, respectively, and $c(\cdot)$ denotes the cost function of the agent.

The principal's objective is to choose a contract $x$ that maximizes the expected utility, taking into account the agent’s optimal response $y$ to the incentives specified by the contract. Formally, the principal’s problem is formulated as:

\begin{align}
	\min_{x \in \mathbb{R}^n, y \in Y} \  & F(x, y) := -\sum_{j=1}^n v(\pi_j - x_j ) P(s_j, y)    \label{pa0}\tag{PA} \\
	\text{s.t.} \quad & f(x, y) := -\sum_{j=1}^n u(x_j) P(s_j, y) + c(y)+\underline{U} \leq 0, \label{IR}\tag{IR}\\
	& y \in \arg\min_{y' \in Y} f(x, y'), \label{IC}\tag{IC} 
\end{align}
where \eqref{IC} represents the incentive compatibility constraint, which ensures that the agent’s chosen action $y$ maximizes the expected utility under the contract. 
\eqref{IR} is the individual rationality constraint, requiring that the agent’s expected utility is no less than the reservation utility $\underline{U}$. This framework was first introduced by Mirrlees \cite{M76}. 

The classical method for solving the principal–agent  problem, known as the first order approach (FOA), involves first solving a relaxed version of the problem, in which the incentive compatibility constraint \eqref{IC} is replaced by its first order  condition. The optimality conditions for this relaxed problem are then derived.
Although previous studies have proposed various sufficient conditions under which the FOA is valid, Mirrlees presents a well-known counterexample  in \cite[Example 1]{M99} showing that the first order approach can fail: the optimal contract may not even be a stationary point of the relaxed problem.

\subsection{Existing approaches for deriving optimality conditions}
The most  common approach is the Karush–Kuhn–Tucker (KKT) approach or the first order  approach if the lower level constraint set  is formed by equality and inequality constraints.  The single level reformulation by the first order approach belongs to the class of mathematical programs with equilibrium constraints (MPEC); see e.g., \cite{GY19,Outrata,Ye00,Yebook,YY} for related research. The popularity of the first order approach is obvious since the resulting reformulations are more computational friendly. However, there are two major drawbacks to apply the first order approach. First, when the lower level problem is not convex, the reformulation by the first order approach is not equivalent to the original bilevel problem. Secondly, even if the lower level problem is convex, the multipliers in the KKT conditions introduce extra variables and consequently the resulting single level reformulation may not be equivalent to the original bilevel program in the sense of local optimality, see Dempe and Dutta \cite{DD}.  Let $(\bar x,\bar y)$ be a local optimal solution of (\ref{BP}). In the case $Y$ is convex and  the lower level objective function $f(x,y)$ is convex in variable $y$,  Ye and Ye \cite{YY} proposed replacing the  constraint $y\in S(x)$ locally around $(\bar x,\bar y)$ by  the generalized equation constraint:
\begin{equation*}
	0\in\nabla_yf(x,y)+ N_{Y}(y),\end{equation*}
where $N_{Y}(y)$ denotes the normal cone  to set $Y$ in convex analysis.                                                    Then under the pseudo-upper Lipschitz continuity/metric subregularity/calmness constraint qualification, they showed in \cite[Theorem 3.2]{YY} that the following first order necessary optimality condition in terms of limiting normal cone holds: 
there exist $((\bar\mu,\bar\nu),\bar\beta)\in\mathbb R^{2m}\times \mathbb R^q_+$  such that
\begin{eqnarray}\label{kkt}
	&& 0=\nabla F(\bar x,\bar y)+(0, \bar \mu) -\nabla (\nabla_yf)(\bar x,\bar y)^T
	\bar\nu+\nabla G(\bar x,\bar y)^T {\bar\beta},\nonumber\\
	&&(\bar\mu,\bar\nu)\in N_{{\rm gph}\, N_{Y}}(\bar y,-\nabla_yf(\bar x,\bar y)),\quad 
	\bar\beta\perp G(\bar x,\bar y),
\end{eqnarray}
where ${{\rm gph}\, N_{Y}}$ denotes  the graph of the  normal cone to $Y$, $N_{{\rm gph}\, N_{Y}}$ denotes the limiting normal cone to the nonconvex set ${{\rm gph}\, N_{Y}}$. Due to the usage of the limiting normal cone, this kind of optimality condition is referred as an Mordukhovich (M-)stationary condition in the literature.
Moreover in the case where  $Y=\{y\in\mathbb R^m \mid  g(y) \leq 0\}$ where  $g:\mathbb R^m\rightarrow\mathbb R^p$ is  convex and certain constraint qualification holds for the  lower level program $P(\bar x)$ at $\bar y$. Let  $\bar \lambda$ be  a multiplier for $P(\bar x)$ at $\bar y$. Then by \cite[Theorem 4.1]{YY}, the above M-stationary condition becomes 
the  existence of  $(\bar\nu,\bar\beta)\in\mathbb R^{m}\times \mathbb R^q_+$ and $ b^*\in \R^p$ such that 
\begin{eqnarray}\label{KKTnew}
	\nonumber
	&&	0=\nabla F(\bar x,\bar y)-\nabla(\nabla_yf+\nabla g^T\bar \lambda)(\bar  x,\bar y)^T  \bar \nu+(0,  \nabla g(\bar y)^T b^*) +\nabla G(\bar x,\bar y)^T {\bar\beta},\ \\
	&& (b^*, \nabla g(\bar y) \nu) \in N_{{\rm gph} N_{\R_-^p}}(g(\bar y), \bar {\lambda}), \quad \bar\beta\perp G(\bar x,\bar y).
\end{eqnarray}

To obtain  optimality conditions for  the general nonconvex bilevel program where the lower level constraint $Y(x)$ is formed by equality and inequality constraints,   Ye and Zhu \cite{YZ95} proposed to replace the constraint $y\in S(x)$  by 
the following  value function constraint 
$$
f(x,y)-V(x)\leq0,\ y\in Y(x),
$$
where $V(x)$ denotes the value function of the lower level problem (\ref{P(x)}):
$$V(x):=\inf\{f(x,y) \mid y\in Y(x)\}.$$
The advantage of the value function approach is that there is no convexity assumption required for the lower level problem.
The resulting single level problem is however a nonsmooth optimization problem  involving the value function which is an implicitly defined function. Normally for an optimization problem, Fritz John condition holds at a local minimizer without any constraint qualification. To obtain the KKT condition one would argue that there is no nonzero abnormal multiplier which is equivalent to saying that the nonsmooth Mangasarian-Fromovitz constraint qualification (MFCQ) holds and so the KKT condition  must hold by the Fritz John condition.  In \cite[Proposition 3.2]{YZ95}, Ye and Zhu discovered that this usual strategy never work for the value function reformulation since there always exist a nonzero abnormal multiplier due to the value function constraint.   They then proposed  the partial calmness condition and studied various sufficient conditions for it in \cite{YZ97}. 
As the partial calmness condition may still be too strong for the single level reformulation, 
Ye and Zhu \cite{YZ} also proposed  the combined approach, which combines the  KKT approach with the value function approach. 
Nevertheless, the partial calmness condition was shown to hold  generically  for the combined program when the upper level variable is one dimension in \cite{KYYZ}. This indicates that the single level program obtained from the combined approach is more likely to satisfy the partial calmness condition.  Recently in \cite{BY}, Bai and Ye showed that even a weaker constraint qualification called the first order sufficient conditions for metric subregularity which is weaker than MFCQ  fails to hold for the bilevel program with equality and inequality constraints. They then derive a directional KKT condition for the single level reformulation of the bilevel program by the value function approach under  the directional calmness condition and proposed the directional quasi-normality as a sufficient condition for the directional calmness condition. Another class of approaches for deriving optimality conditions is based on local reduction methods via nondegenerate local minimizers~\cite{Dempe2009,Jongen2012}. A recent contribution in this direction is the work of Bolte et al.~\cite{BLPV26}, which introduces the Morse–parametric qualification condition under which the bilevel problem admits a mixed-integer nonlinear programming relaxation. 

\subsection{Our new approach}

As discussed above, the first order approach  and the approaches involving the value function have their own advantages and disadvantages. The reformulation involving 
the value function can apply to the  bilevel program with a nonconvex lower level program but the resulting single level reformulation is an optimization problem involving a value function constraint which may be difficult to handle. The first order approach does not have the above difficulties but it requires  convexity of the lower level program. In our new approach we will try to modify the first order approach proposed by Ye and Ye \cite{YY} so that it can be applied to (\ref{BP}) with nonconvex lower level.

To get an insight into the issue and our approach, let us examine the following bilevel program of Mirrlees \cite[Example 1]{M99}:
\begin{equation*}
	\begin{aligned}
		\min &\  F(x, y):=(x-2)^2+(y-1)^2 \\
		\mathrm { s.t. }\quad 
		& y \text { minimizes }  f(x, y)=-x e^{-(y+1)^2}-e^{-(y-1)^2}.
	\end{aligned}
\end{equation*}
\begin{figure}[H]
	\centering
	\includegraphics[width=0.5\linewidth]{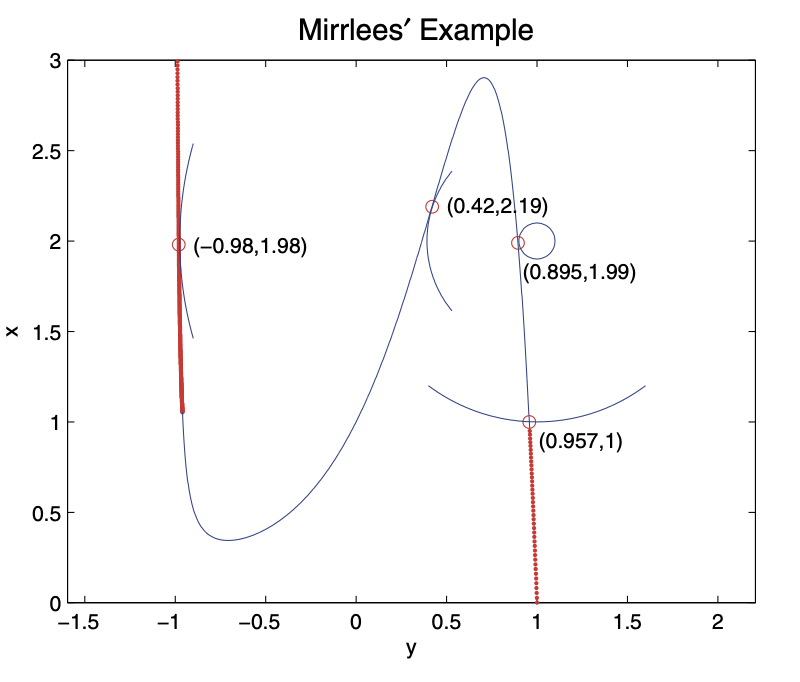}
	\caption{Mirrlees example}
	\label{Figure M}
\end{figure}
In Figure \ref{Figure M}, the blue curve is the graph of the first order condition
$\nabla_y f(x,y)=0$ and 
the graph of the lower level solution mapping $S(x)$ is the red  disconnected curve with a jump at $\bar{x}=1$. Since the feasible region of the bilevel program is the red curve, it is clear that the point $(\bar{x},\bar{y})=(1,0.957)$ is the global optimal solution. However this point is not even a stationary point of the single level program 
which has the blue curve as the constraint region. Take a look at Figure \ref{Figure M}, in the neighborhood of the point $(1,0.957)$, for $x<1$, the blue curve coincides with the red curve but for $x>1$, the blue curve does not coincide with the red curve.
Consequently,  
the first order condition is not a locally equivalent reformulation of the constraint $y\in S(x)$ at the point $(\bar{x},\bar{y})$ since one cannot find a neighborhood around the point $(\bar x,\bar y)$ at which the restriction of the graph of the first order condition is the same as the graph of the solution map.  That is why  the first order approach fails at point $(1,0.957)$.
However from the figure in  the region $(-\infty, 1]\times (0.5, +\infty) $,  the blue curve coincides with the red curve. That is, the  constraint 
\begin{eqnarray*}
	\nabla_y f(x,y)=0, x\leq 1, y\geq 0.5
\end{eqnarray*} and \begin{eqnarray*}
	y\in S(x), x\leq 1, y\geq  0.5
\end{eqnarray*}
are equivalent.
So the global optimal solution of the bilevel program  $(\bar{x},\bar{y})=(1,0.957)$ must be an optimal solution of the following modified  first order reformulation  
\begin{eqnarray*}
	\min && \ F(x,y)\\
	\mathrm { s.t. }  &&\  \nabla_y f(x,y)=0, \ x\leq 1, \ y \geq 0.5.
\end{eqnarray*}

Note that the region $(-\infty, 1]\times [0.5, +\infty) $ includes the set $(-\infty, 1]\times (\bar y-0.4,\bar y+0.4)$ which is a directional neighborhood of  $(1,0.957)$  (see Definition \ref{dn} and the comments after the definition) in direction $(u,v)$ with  $u<0, v\in \mathbb R$. Hence the above modified first order reformulation is a directional first order reformulation at $(\bar x,\bar y)$ in direction $u<0, v\in \mathbb R$.

Motivated by Mirrlees' example, in this paper we propose a directional first order approach for bilevel program (\ref{BP}). Since  $Y$ is a closed and not necessarily convex set, $f$ is smooth and  $ y \in S(x)$,  by the first order necessary optimality condition,  the generalized equation holds
$$0\in \nabla_y f(x,y)+ \widehat{N}_Y(y), $$
where $\widehat N_{Y}$ denotes  the regular normal cone to $Y$ at $y$. Using the graph of the regular normal cone mapping, the generalized equation constraint   is equivalent to the set  constraint
\begin{equation}\label{setc}  \left(y,-\nabla_yf(x,y)\right)\in{\rm gph}\,\widehat N_{Y}.\end{equation} 
The first order approach in terms of the regular normal cone amounts to replace the constraint $y\in S(x)$ by the set constraint (\ref{setc}) and solve the following   set constrained optimization problem: 
\begin{equation}\tag{SCOP}
	\begin{aligned}
		\min_{x,y}\  &F(x,y)\\
		{\rm s.t.}\  & G(x,y)\leq0,\
		\left(y,-\nabla_yf(x,y)\right)\in{\rm gph}\,\widehat N_{Y}.
	\end{aligned} 
\end{equation}
Since the lower level is nonconvex,  problem (SCOP) is not locally equivalent to problem (\ref{BP}).

Motivated by the above discussion, given an optimal solution $(\bar x,\bar y)$  of (\ref{BP}) and a direction $u\in\mathbb R^n$, we consider the following set constrained optimization problem:
\begin{equation}\tag{SCOP\textsubscript{u}}
	\begin{aligned}
		\min_{x,y}\  &F(x,y)\\
		{\rm s.t.}\  & G(x,y)\leq0,\
		\left(y,-\nabla_yf(x,y)\right)\in{\rm gph}\,\widehat N_{Y},\\
		&(x,y)\in(\bar x,\bar y)+{\cal V}_{\varepsilon,\delta}(u)\times\varepsilon\mathbb B,
	\end{aligned} 
\end{equation}
where    ${\cal V}_{\varepsilon,\delta}(u)$ denotes the directional neighborhood of $0$ in direction $u$; see Definition \ref{dn}. When $u=0$, {\rm (SCOP$_{u}$)} recovers the set constrained optimization problem (SCOP) restricted on a classical neighborhood of $(\bar x,\bar y)$. As discussed above, if the lower level problem is nonconvex, the first order reformulation (SCOP) is usually not an equivalent local reformulation to the bilevel program (\ref{BP}). For  {\rm (SCOP$_{u}$)} to be a  directional first order reformulation for (\ref{BP}), we need to impose some conditions. 
Define the lower level stationary solution map as $$S_{\mathrm{FO}}(x):=\{y\in Y \mid 0\in\nabla_yf(x,y)+\widehat N_Y(y)\}.$$ 
For an $x$, $S_{\mathrm{FO}}(x)$ may not be a singleton. But very often for $x\in U$,  a neighborhood of  $\bar x$, there exists $V$, a neighborhood of $\bar y$ so that  $S_{\mathrm{FO}}(x)\cap V$ is single-valued for all $x\in U$. This property is called a single-valued localization. For example from Figure \ref{Figure M}, we can see that the lower level stationary solution map for Mirrlees' problem has a single-valued localization at each point near the global optimal solution. Suppose that the stationary solution map $S_{\mathrm{FO}}(x)$ has a single-valued localization. When will the constraints
$0\in \nabla_y f(x,y)+ \widehat{N}_Y(y)$
and $ y\in S(x)$  coincide on a directional neighborhood? Since any optimal solution must be a stationary point, if the stationary map $S_{\mathrm{FO}}(x)$ has a single-valued localization, then  $S(x)=S_{\mathrm{FO}}(x)$ provided that  $S(x)\cap V $ is nonempty for $x$ in a directional neighborhood of $\bar x$. For example from Figure \ref{Figure M}, $S(x)\cap V$ is nonempty for $x\in (-\infty, 1]$ which is a directional neighborhood of $\bar x$ in direction $u<0$ and  $V=(\bar y-0.4, \bar y+ 0.4 )$ is a directional neighborhood of $\bar y$.
Based on these observations, we have shown that for problem (\ref{BP}), under the condition that $S_{\mathrm{FO}}(x)$ has a single-valued localization and the solution map $S(x)$ is inner semi-continuous at $(\bar x,\bar y)$ in direction $u$, problem   {\rm (SCOP$_{u}$)} is a directional local reformulation to the bilevel program (\ref{BP}). 


A natural question arises: for the bilevel program (\ref{BP}) with a nonconvex lower level, suppose {\rm (SCOP$_{u}$)} is a directional local reformulation   at $(\bar x,\bar y)$ which is a local optimal solution of (\ref{BP}), does the M-stationary condition (\ref{kkt}) with ${N}_Y$ replaced by $\widehat{N}_Y$ still hold under the metric subregularity condition?
In this paper we give this question a positive answer. We show  in Theorem \ref{opt} that if {\rm (SCOP$_{u}$)} is a directional local reformulation to the bilevel program (\ref{BP}) at $(\bar x,\bar y)$ and  there is  $v\in \mathbb R^m$ such that $(u,v)$ lies in the tangent cone of the feasible region of problem (SCOP) and $\nabla F
(\bar x, \bar y)(u,v)=0$, then a directional version of the M-stationary condition holds. In fact the directional M-stationary condition is sharper when $u$ is not equal to zero. 

Our result even gives a new insight into the simpliest case where there are no constraints. For this simple case if the lower level objective function is not convex in variable $y$, 
then a local optimal solution $(\bar x, \bar y)$ may not satisfy the KKT condition of the first order reformulation
$$ \min_{x,y} F(x,y) \quad \mbox{s.t.}\quad\nabla_y f(x,y)=0,$$ under a usual constraint qualification. However according to our directional first order approach, it still satisfies the KKT condition of the first order reformulation under a usual constraint qualification provided that $\nabla_{yy}^2 f(\bar x,\bar y) $ is nonsingular  and the inf-compactness condition holds at $\bar x$,  and either $\bar y$ is the unique point in $
S(\bar x)$  or $\bar y$ is not the unique point in $
S(\bar x)$ but there exists $u \in \mathbb{R}^n$ such that the directional derivative inequality
$$u^T \nabla_xf(\bar x,\bar y)<u^T \nabla_xf(\bar x, y) 
\quad  
\forall y\in S(\bar x)\backslash\{\bar y\},$$ holds and there exists a direction $v$ such that 
$$\nabla \nabla_y f(\bar x,\bar y)(u,v)=0,\quad \nabla F(\bar x,\bar y)(u,v)=0.$$


%

In summary, due to its simplicity, the first order approach has been proven to be useful in both theory and computations. However the validity of the first order approach relies on the convexity of the lower level program. In this paper we  investigate the possibility of utilizing the (directional) first order approach when the lower level is nonconvex.
The main contribution of the paper is as follows:
\begin{itemize}
	\item  For \eqref{BP} without lower level convexity, we propose a local directional characterization by the first order condition for the lower level program,  under the single-valued directional localization of the first order stationary map and the directional inner semi-continuity of the lower level solution map; see Theorem \ref{genthm} and Remarks \ref{R3.1} and \ref{R3.2}.
	\item  We   propose  new and verifiable sufficient conditions for the (directional) inner semi-continuity of the lower level solution map; see Propositions \ref{suffdinc} and \ref{insc}. Since there are very few sufficient conditions for inner semi-continuity of a solution map in the theory of parametric programs,  these conditions are of independent interest.
	\item The principal-agent problem is an important problem in Economics. In this paper we show that the directional first order approach is applicable to a large class of principal-agent problems without the lower-level convexity assumption used in the classical theory; see Section~\ref{PA_application}.
\end{itemize}

\subsection{Organization}
We organize the paper as following. In Section~\ref{pre}, we provide the notations, preliminaries and preliminary results. In Section~\ref{Equivreform} under reasonable conditions, we show that the lower level program can be equivalently characterized by its first order condition over a directional neighborhood. In Section~\ref{firstcondition}, we propose the directional SCOP reformulation for (\ref{BP}) and establish necessary optimality conditions for the directional SCOP reformulation of (\ref{BP}).  Finally, we present an example where the classical first order approach fails while our directional one is applicable.

\section{Preliminaries}\label{pre}

In this section, we review  some basic concepts and results in variational analysis, which will be used later on. For more details see e.g. \cite{BS,DonRock,Long,RW}. 

Denote by $\|\cdot\|$ the norm of Euclidean spaces and by $\langle\cdot,\cdot\rangle$ the inner product of Euclidean spaces.
Let $\Omega$ be a set. By $x^k\xrightarrow{\Omega}\bar{x}$ we mean $x^k\rightarrow\bar{x}$ and $x^k\in \Omega$ for each $k$.  By $x^k\xrightarrow{u}\bar x$ where $u$ is a vector, we mean that the sequence $\{x^k\}$ approaches $\bar x$ in direction $u$, i.e., there exist $t_k\downarrow 0, u^k\rightarrow u$ such that $x^k=\bar x+ t_k u^k$. By $o(t)$, we mean $\lim_{t\rightarrow0}\frac{o(t)}{t}=0$.
We denote by $\mathbb B$, $\bar{\mathbb B}$, $\mathbb S$  the open unit ball, the closed unit ball and 
the unit sphere, respectively. $\mathbb B_\delta(\bar z)$ denotes the open unit ball centered at $\bar z$ with radius $\delta$.
We denote by ${\rm co}\,\Omega$, $\mathrm{cl}\,\Omega$, $\mathrm{span}\, \Omega$ and $\Omega^\circ:=\{y \mid y^Tx\leq0\ \forall x\in \Omega\}$ the convex hull, the closure, the span (the smallest subspace containing $\Omega$) and the polar cone of a set $\Omega$, respectively. The distance from a point $x$ to a set $\Omega$ is denoted by  ${\rm dist}(x,\Omega):=\inf\{\|x-y\|\,|\,y\in\Omega\}$ and the indicator function of set $\Omega$  is denoted by  $\delta_\Omega$. For a single-valued map $\phi:\mathbb R^n\rightarrow\mathbb R^m$, we denote by $\nabla \phi(x)\in \mathbb{R}^{m\times n}$   the Jacobian matrix of $\phi$  at $x$ and for a function $\varphi:\mathbb R^n\rightarrow\mathbb R$, we denote by $\nabla \varphi(x)$ both the gradient and the Jacobian of $\varphi$  at $x$. 
For a set-valued map $\Phi:\mathbb R^n\rightrightarrows\mathbb R^m$ the graph of $\Phi$ is defined by  ${\rm gph}\,\Phi:=\{(x,y) \mid y\in \Phi(x)\}$. 

\begin{defn}[Tangent Cone and Normal Cone](see, e.g., \cite{RW})\label{geometry}
	Given a set $\Omega\subseteq\mathbb R^n$ and a point $\bar{x}\in \Omega$, the contingent cone to $\Omega$ at $\bar{x}$ is defined as
	$$T_\Omega(\bar{x}):=\left \{d\in\mathbb R^n \,|\, \exists t_k\downarrow0, d_k\rightarrow d\ \mbox{s.t.} \ \bar{x}+t_kd_k\in \Omega\   \forall k\right \}.$$
	The Fr\'echet normal cone and the limiting normal cone to $\Omega$ at $\bar{x}$ are defined as
	\begin{align*}
		\widehat{N}_\Omega(\bar{x})&:=\left \{ \zeta\in \mathbb{R}^n \,\Big|\, \langle \zeta ,x-\bar{x}\rangle \leq o(\|x-\bar x\|) \ \forall x\in \Omega \right \},\\
		N_\Omega(\bar{x})&:=\left \{\zeta\in \mathbb{R}^n\,\Big|\, \exists \ x_k\xrightarrow{\Omega}\bar{x},\ \zeta_k{\rightarrow}\zeta\ \mbox{s.t.} \ \zeta_k\in\widehat{N}_\Omega(x_k)\ \forall k\right \},\end{align*}
	respectively. Moreover, $\widehat{N}_\Omega(z):=\emptyset$ if $z\notin \Omega$.
\end{defn}

\begin{defn}[Directional Normal Cone] (\cite[Definition 2]{Gfr13}  or \cite[Definition 2.3]{GM}).\label{dnorm} 
	Given a set $\Omega \subseteq \mathbb{R}^n$, a point $\bar x \in \Omega$ and a direction $d\in\mathbb{R}^n$, the limiting normal cone to $\Omega$ at $\bar{x}$ in direction $d$ is defined by
	\[N_\Omega(\bar{x};d):=\left \{\zeta \in \mathbb{R}^n\,\Big| \,\exists \ t_k\downarrow0, d_k\rightarrow d, \zeta_k\rightarrow\zeta  \mbox{ s.t. } \zeta_k\in \widehat{N}_\Omega(\bar{x}+t_kd_k)\ \forall k \right \}.\]
\end{defn}
It is obvious that $N_{\Omega}(\bar x; 0)=N_{\Omega}(\bar x)$, $N_{\Omega}(\bar x; d)=\emptyset$ if $d \not \in T_\Omega(\bar x)$ and $N_{\Omega}(\bar x;d)\subseteq N_\Omega(\bar x)$. It is also obvious that for all $d\in T_\Omega(\bar x) \setminus T_{bd \Omega}(\bar x)$, one has $N_\Omega(\bar x;d)=\{0\}$.
Moreover when $\Omega$ is convex, by \cite[Theorem 3.2]{Long} the directional and the classical normal cone have the following relationship which extends the corresponding  result in \cite[Lemma 2.1]{Gfr14} that concerns finite dimensional spaces
\begin{equation}
	N_\Omega(\bar x;d)=N_\Omega(\bar x)\cap \{d\}^\perp  \qquad \forall d\in T_\Omega(\bar x).\label{convNormal}
\end{equation}

%

\begin{defn}[Directional Neighborhood]\label{dn} (\cite[formula (7)]{Gfr13}).
	Given a direction $d$ in $\mathbb{R}^n$, and positive numbers $\varepsilon,\delta>0$, the directional neighborhood of {the origin} in direction $d$ is a set defined by 
	\begin{equation*}
		{\cal V}_{\varepsilon,\delta}(d)
		:=\left \{ 
		\begin{array}{ll}
			\{0\} \cup \left \{ z \in \varepsilon\mathbb{B}\setminus \{0\} \,\Big|\,
			\left\|\frac{z}{\|z\|} -\frac{d}{\|d\|} \right\|
			\leq \delta \right \} & \mbox{ if } d\not =0,\\
			\varepsilon\mathbb{B} &  \mbox{ if } d =0.
		\end{array} \right . \label{eqn dneighbor}
	\end{equation*}
	For $\bar z \in \mathbb{R}^n$, we call $\bar z+{\cal V}_{\varepsilon,\delta}(d)$ a directional neighborhood of $\bar z$ in direction $d$.
\end{defn}
It is easy to see that the directional neighborhood of the origin in direction $d=0$ is just the open ball $\varepsilon\mathbb{B}$ and the directional neighborhood of the origin in a nonzero direction $d\neq0$ is a section of the open  ball  $\varepsilon\mathbb{B}$ with the central angle determined by $\delta$; see \cite[Figure 1]{BY23}.


We now give the definition of directional metric subregularity.
\begin{defn}[Directional Metric Subregularity] (\cite[Definition 1]{Gfr13})
	Let $\Phi:\mathbb{R}^n\\ \rightrightarrows\mathbb{R}^m$ be a set-valued map and $(\bar{x},\bar{y})\in {\rm gph}\, \Phi$.
	Given a direction $u\in\mathbb{R}^n$, $\Phi$ is said to be metrically subregular in direction $u$ at $(\bar{x},\bar{y})$, if  there are positive reals $\varepsilon,\delta,$ and $\kappa$ such that
	\begin{equation*}
		{{\rm dist}}(x,\Phi^{-1}(\bar{y}))\leq\kappa {{\rm dist}}(\bar{y}, \Phi(x)) \quad \forall  x\in\bar{x}+{\cal V}_{\varepsilon,\delta}(u).
	\end{equation*}
\end{defn}
If $u=0$ in the above definition, then we say that the set-valued map $\Phi$ is metrically subregular at $(\bar x,\bar y)$. {It is known that the metric subregularity of $\Phi$ at $(\bar x,\bar y)$ is equivalent to the calmness of $\Phi^{-1}$ at $(\bar y,\bar x)$ (see \cite[Theorem 3H.3]{DonRock}). 
	
	The following proposition lists some existing sufficient conditions for the (directional) metric subregularity, hence can be used as constraint qualifications for the directional necessary optimality condition. Interested readers are referred to \cite{Gfr13,Robinson} for more details. We say a set is polyhedral convex if it can be expressed as the solution set of finitely many linear equations and inequalities. And we say a set is polyhedral if it is a union of finitely many polyhedral convex sets.
	\begin{prop}\label{qp} Consider a set constraint system $\phi(z)\in\Omega$, where $\phi(z):\mathbb R^p\rightarrow\mathbb R^s$ is continuously differentiable and $\Omega\subseteq\mathbb R^s$ is locally closed. Let $\phi (\bar z) \in \Omega$ and $d\in \mathbb{R}^p$ with $d$  satisfying   the condition $\nabla \phi(\bar z) d\in T_\Omega (\phi(\bar z))$.
		\begin{itemize}
			\item Suppose that the no nonzero abnormal multiplier constraint qualification (NNAMCQ) holds at $\bar z$, i.e.,
			\begin{equation*}
				0= \nabla\phi(\bar z)^T\zeta \mbox{ and } \zeta\in N_\Omega(\phi(\bar z))   \Longrightarrow \zeta=0.
			\end{equation*}
			Then the set-valued map $-\phi(z)+\Omega$ is metrically subregular at $(\bar z,0)$.
			\item   Suppose that the first order sufficient condition for metric subregularity (FOSCMS)  holds  at $(\bar z,0)$ in direction $d$, i.e., there exists no $\zeta
			\neq 0$ satisfying 
			\begin{equation*}\label{FOSCMScon}
				0= \nabla\phi(\bar z)^T\zeta~\mbox{ and }~	\zeta\in N_\Omega( \phi(\bar z);\nabla\phi(\bar z)d).
			\end{equation*}	
			Then the set-valued map $-\phi(z)+\Omega$ is metrically subregular at $(\bar z,0)$ in direction $d$.
			\item Suppose that the affine+polyhedral condition holds at $\bar z$, i.e., $\phi$ is affine and $\Omega$ is polyhedral.
			Then the set-valued map $-\phi(z)+\Omega$ is metrically subregular at $(\bar z,0)$.
		\end{itemize}
	\end{prop}

	Define the linearized cone $L(\bar z):=\{d\in\mathbb R^p \mid \nabla\phi(\bar z)d\in T_\Omega(\bar z)\}$. It is known that $T_{\phi^{-1}(\Omega)}(\bar z)\subseteq L(\bar z)$. Under the directional metric subregularity, these two cones have the following relationship. This result will be used in Section 4.
	\begin{prop}\label{dirAbadie}
		Given any direction $d\in L(\bar z)$. Suppose metric subregularity for $-\phi(z)+\Omega$  holds at $(\bar z,0)$ in direction $d$. Then $d\in T_{\phi^{-1}(\Omega)}(\bar z)$.
	\end{prop}
	\begin{proof}
	By the assumed directional metric subregularity, there exist $\kappa>0$ such that for any sequences $t_k\downarrow0$ and $d^k\rightarrow d$, one has for sufficiently large $k$
	\begin{align*}
		{\rm dist}(\bar z+t_kd^k,\phi^{-1}(\Omega))&\leq\kappa {\rm dist}(\phi(\bar z+t_kd^k),\Omega)\\
		&=\kappa {\rm dist}(\phi(\bar z)+t_k\nabla \phi(\bar z)d+o(t_k),\Omega). 
	\end{align*}
	Since $d\in L(\bar z)$, by Definition \ref{geometry}, taking appropriate $\{t_k\},\{d^k\}$, we have $\phi(\bar z)+t_k\nabla \phi(\bar z)d+o(t_k)\in \Omega$. Consequently,
	\begin{align*}
		{\rm dist}(\bar z+t_kd^k,\phi^{-1}(\Omega))\leq\kappa {\rm dist}(\phi(\bar z)+t_k\nabla \phi(\bar z)d+o(t_k),\Omega)=o(t_k).
	\end{align*}
	This implies that $d\in T_{\phi^{-1}(\Omega)}(\bar z)$.
	\end{proof}
	From the above proposition, if the metric subregularity for $-\phi(z)+\Omega$  holds at $(\bar z,0)$ in all directions $d\in L(\bar z)$, then $T_{\phi^{-1}(\Omega)}(\bar z)\supseteq L(\bar z)$. Equivalently, the classical metric subregularity implies the so-called Abadie constraint qualification.

	\begin{defn}[Directional Inner Semi-continuity](\cite[Definition 4.4]{BY})\label{incont}
		Let $S:\mathbb{R}^n\rightrightarrows\mathbb{R}^m$ be a set-valued map. 
		Given $\bar y\in S(\bar x)$, we say $S(x)$ is inner semi-continuous at $(\bar x,\bar y)$ in direction $u$  provided that, for any sequences $t_k\downarrow0$ and $u^k\rightarrow u$, there exists a sequence $y^k\in S(\bar x+t_ku^k)$ converging to $\bar y$. When $u=0$ in the above, we say that $S(x)$ is inner semi-continuous at $(\bar x,\bar y)$.
	\end{defn}
	\begin{remark}
		Using the concept of directional neighborhood, the directional inner semi-continuity can be equivalently defined as: there exist positive scalars $\varepsilon,\delta$ such that for any sequence $\{x^k\}\subseteq\bar x+\mathcal V_{\varepsilon,\delta}(u)$ converging to $\bar x$, there exists a sequence $y^k\in S(x^k)$ converging to $\bar y$.
	\end{remark}
	\section{Local  characterization of the lower level program constraint}\label{Equivreform}
	
	This section serves to present a local directional characterization for the constraint $y\in S(x)$ by the first order condition. To this end, first under very weak conditions, we show that the two constraints
	\[
	y\in S(x)~~\mbox{and}~~(y,-\nabla_yf(x,y))\in{\rm gph}\,\widehat{N}_{Y}
	\]
	are equivalent over a directional neighborhood of an optimal solution point. Then we provide various sufficient conditions to verify this equivalence. Finally, we separate our study into two cases and give more detailed discussions.
	\begin{thm}\label{genthm}
		Let $\bar y\in S(\bar x)$ be a solution to lower level program {\rm (P$_x$)} and $u\in\mathbb{R}^n$ be a direction. Suppose there exist positive scalars $\varepsilon_x,\varepsilon_y,\delta$ such that the following two conditions hold,
		\begin{itemize}
			\item[(i)]  $S_{\mathrm{FO}}(x)\cap(\bar y+\varepsilon_y\mathbb B)$ is single-valued over $\bar x+\mathcal{V}_{\varepsilon_x,\delta}(u)$,
			\item[(ii)] $S(x)\cap(\bar y+\varepsilon_y\mathbb B)\neq\emptyset$ over $\bar x+\mathcal{V}_{\varepsilon_x,\delta}(u).$
		\end{itemize}
		Then the stationary map $S_{\mathrm{FO}}$ and solution map $S$ coincide over $\bar x+\mathcal{V}_{\varepsilon_x,\delta}(u)$, that is,
		\begin{equation}\label{equiv}
			S_{\mathrm{FO}}(x)\cap(\bar y+\varepsilon_y\mathbb B)=S(x)\cap(\bar y+\varepsilon_y\mathbb B)\neq\emptyset~~~\forall x\in\bar x+\mathcal{V}_{\varepsilon_x,\delta}(u).
		\end{equation}
		Consequently, the two constraints
		\[
		y\in S(x)~~\mbox{and}~~(y,-\nabla_yf(x,y))\in{\rm gph}\,\widehat{N}_{Y}
		\]
		are equivalent over the directional neighborhood $(\bar x,\bar y)+\mathcal{V}_{\varepsilon_x,\delta}(u)\times\varepsilon_y\mathbb B$. Particularly, in case $u=0$, these two constraints are equivalent on a neighborhood of $(\bar x,\bar y)$ in the classical sense.
	\end{thm}
	\begin{proof}		
	For each $y\in S(x)$, it is well-known that the following stationarity condition holds,
	\[
	0\in\nabla_yf(x,y)+\widehat{N}_Y(y).
	\] 
	Then together with condition (ii), we have    \begin{equation*}
		\emptyset\neq S(x)\cap(\bar y+\varepsilon_y\mathbb B)\subseteq S_{\rm FO}(x)\cap(\bar y+\varepsilon_y\mathbb B)~~\forall x\in\bar x+\mathcal{V}_{\varepsilon_x,\delta}(u).
	\end{equation*}
	
	On the other hand, by condition (i), map $S_{\rm FO}(x)$ has a single-valued localization near $(\bar x,\bar y)$ in direction $u$. Then for any $(x,y)\in(\bar x,\bar y)+\mathcal{V}_{\varepsilon_x,\delta}(u)\times\varepsilon_y\mathbb B$ satisfying that $0\in\nabla_yf(x,y)+\widehat{N}_Y(y)$, we have $y\in S(x)$. Indeed, otherwise there exists some $(x',y')\in (\bar x,\bar y)+\mathcal{V}_{\varepsilon_x,\delta}(u)\times\varepsilon_y\mathbb B$ such that $0\in\nabla_yf(x',y')+\widehat{N}_Y(y')$ and $y'\notin S(x')$. Consequently, $S_{\rm FO}(x')\supseteq S(x')\cup\{y'\}$ is not a singleton, violating condition (i).
	
	In conclusion, we establish the equality in (\ref{equiv}), and therefore, the equivalence between the constraints
	\[
	y\in S(x)~~\mbox{and}~~(y,-\nabla_yf(x,y))\in{\rm gph}\,\widehat{N}_{Y}.
	\]
	\end{proof}
	
	\begin{remark}\label{R3.1}	
		Recall that  $S_{\rm FO}(x)$ has a single-valued Lipschitzian localization near $(\bar x,\bar y)$ if there are neighborhoods $U$ of $\bar x$, $V$ of $\bar y$ and a Lipschitz continuous mapping $S_{loc}: U \rightarrow V$ such that
		$$\{S_{loc}(x)\} =S_{FO}(x) \cap V \quad \forall x\in U.$$
		In this spirit, we can say that condition (i) means that the solution mapping of the generalized equation $0\in\nabla_yf(x,y)+{\widehat N}_Y(y)$ has a single-valued directional localization near $(\bar x,\bar y)$ in direction $u$. 
	\end{remark}
	\begin{remark}\label{R3.2}
		Condition (i) assumed that the solution mapping of the variational inequality $0\in\nabla_yf(x,y)+\widehat{N}_Y(y)$ has a single-valued localization around $(\bar x,\bar y)$. According to \cite{DonRock}, when set $Y$ is polyhedral convex, or $\nabla_yf(x,y)$ admits strong monotonicity, sufficient conditions for the existence of single-valued continuous localization of the variational inequality can be established; see e.g. \cite[Theorems 2E.8 and 2F.7]{DonRock}. 
		Condition (ii) is weaker than the directional inner semi-continuity of $S$ at $(\bar x,\bar y)$ in Definition \ref{incont}.
	\end{remark}

	\subsection{Sufficient conditions for single-valued localization of the first order stationary map}\label{subset3.1}
	
	In the following, we establish verifiable sufficient conditions to obtain results in Theorem \ref{genthm}(i). 
	The following strict graphical derivative for set-valued mappings will be used.
	
	\begin{defn}[Strict Graphical Derivative for a Set-Valued Mapping](\cite[Section 4.4]{DonRock})\label{strctgd}
		For a set-valued mapping $\psi$ : $\mathbb{R}^n \rightrightarrows \mathbb{R}^m$ at $\bar{x}$ for $\bar{y}$, where $\bar{y} \in \psi(\bar{x})$, we define its 
		strict graphical derivative as
		\begin{eqnarray*}
			&&D_* \psi(\bar{x} \mid \bar{y})(u)\\
			&&:=\left\{v\in \mathbb{R}^m\, \left|
			\begin{array}{ll}
				&\exists\, (x^k, y^k) \in \operatorname{gph} \psi,(x^k, y^k) \rightarrow(\bar{x}, \bar{y}), t_k\downarrow 0, (u^k,v^k) \rightarrow(u, v)\\
				& \mathrm{s.t.}\ (x^k+t_ku^k,y^k+t_kv^k) \in \operatorname{gph} \psi 
			\end{array}\right.\right\}.
		\end{eqnarray*}
	\end{defn}

	Define the critical cone to $Y$ with respect to $-\nabla_yf(\bar x,\bar y)$ at $\bar y$ by
	\[
	K_Y(\bar y,-\nabla_yf(\bar x,\bar y)):=T_Y(\bar y)\cap\{\nabla_yf(\bar x,\bar y)\}^\perp.
	\]

	\begin{thm}\label{IF} Let $\bar y\in S_{\rm FO}(\bar x)$ and $f$ is twice continuously differentiable at $(\bar x,\bar y)$.  Then under one of the following conditions, $S_{\rm FO}$ has a Lipschitz single-valued localization near $(\bar x,\bar y)$:
		\begin{itemize}
			\item [(i)] $Y$ is locally closed, $\operatorname{gph}{\widehat N}_Y$ is locally closed, 
			$S_{\rm FO}(x)$ is inner semi-continuous at $(\bar x,\bar y)$,
			and  $w\in \mathbb{R}^m$,
			\begin{equation}\label{smrnew}
				(w, -\nabla_{yy}^2 f(\bar x,\bar y)(w))\in \operatorname{gph}D_{*}{\widehat N}_{Y}(\bar{y} \mid -\nabla_yf(\bar x,\bar y))\Rightarrow\ w=0
			\end{equation}
			holds. 
			\item[(ii)] $\bar y\in \operatorname{int}Y$ and $\nabla_{yy}^2 f(\bar x,\bar y)$ is nonsingular.
			\item[(iii)] $Y$ is nonempty closed convex and the following condition holds: there exists $\mu>0$ such that 
			$$\langle w,\nabla_{yy}^2 f(\bar x,\bar y)w\rangle \geq \mu \|w\|^2 \qquad \forall w \in Y-Y.$$
			\item[(iv)] $Y$ is polyhedral convex and, 
			\begin{itemize}
				\item[(a)] either the second order sufficient optimality condition holds:
				$$\langle w,\nabla_{yy}^2 f(\bar x,\bar y)w\rangle >0 \quad \forall 0\not =w \in \{Y-\bar y\} \cap \{\nabla_y f(\bar x,\bar y)\}^\perp,$$
				\item[(b)] or 
				$S_{\rm FO}(x)$ is inner semi-continuous at $(\bar x,\bar y)$, and for any $w \in \{Y-\bar y\} \cap \{\nabla_y f(\bar x,\bar y)\}^\perp$,
				\begin{eqnarray*}
					\left\{
					\begin{array}{ll}
						-\nabla_{yy}^2 f(\bar x,\bar y)w\in {\rm cl}\, (\{Y-\bar y\}^\circ+\mathrm{span}\,\{\nabla_yf(\bar x,\bar y)\}),\\	
						\langle w,-\nabla_{yy}^2 f(\bar x,\bar y)w\rangle=0
					\end{array}
					\right.
					\implies w=0.
				\end{eqnarray*}
			\end{itemize}
		\end{itemize}
	\end{thm}
	\begin{proof}
	(i)  From the definition of strict graphical derivative, the set-valued map $\nabla_{yy}^2 f(\bar x,\bar y)(w)+D_{*}{\widehat N}_{Y}(\bar{y} \mid -\nabla_yf(\bar x,\bar y))(w)$ is positively homogeneous, hence condition (6) in \cite[Corollary 4D.2]{DonRock} is equivalent to the condition that, there exists no $w\neq0$ such that
	\[
	0\in\nabla_{yy}^2 f(\bar x,\bar y)(w)+D_{*}{\widehat N}_{Y}(\bar{y} \mid -\nabla_yf(\bar x,\bar y))(w),
	\]
	which is equivalent to condition (\ref{smrnew}) in this theorem.
	
	(ii) In case $\bar y\in \operatorname{int}Y$, $S\mathrm{_{FO}}(x):=\{y \mid \nabla_yf(x,y)=0\}$. Since $\nabla_{yy}^2 f(\bar x,\bar y)$ is nonsingular, the existence of Lipschitz continuous single-valued localization of $S\mathrm{_{FO}}$ is well-known from implicit function theorem.
	
	(iii) The conclusion follows from \cite[Theorem 2F.7]{DonRock}.
	
	(iv-a) Denote by 
	$$
	A:=\nabla^2_{yy}f(\bar x,\bar y)~\mbox{and}~K:=K_Y(\bar y;-\nabla_yf(\bar x,\bar y)).
	$$
	According to \cite[Theorem 2E.8]{DonRock}, if for each $u\in\mathbb R^m$ the auxiliary variational inequality $Aw-u+N_K(w)\ni 0$ has a unique solution, then $S_{\rm FO}$ has a Lipschitz continuous single-valued localization around $(\bar x,\bar y)$. 
	
	Consider the affine-polyhedral variational inequality $Aw-u+N_K(w)\ni 0$. Denote by
	$$
	K^+:=\operatorname{span}K~\mbox{and}~K^-:=\operatorname{lin}K,
	$$ the   smallest/largest subspace including/included in $K$. By \cite[Theorem 2E.6]{DonRock}, if the following condition holds,
	\begin{equation}\label{2cond}
		w\in K^+, ~ \nabla^2_{yy}f(\bar x,\bar y)w\perp K^-, ~ \langle w,\nabla^2_{yy}f(\bar x,\bar y)w\rangle\leq0~\implies~w=0,
	\end{equation} then variational inequality $Aw-u+N_K(w)\ni 0$ has a unique solution for any $u$, hence $S_{\rm FO}$ has a Lipschitz continuous single-valued localization around $(\bar x,\bar y)$. 
	
	By definition, since $K^-\subseteq K\subseteq K^+$, condition (\ref{2cond}) can be implied by the following second order sufficient optimality condition
	\[
	\langle w,\nabla^2_{yy}f(\bar x,\bar y)w\rangle>0\quad
	\forall w\in K\backslash\{0\}.
	\]
	
	Since $Y$ is a polyhedral convex set, $T_Y(\bar y)=\operatorname{cone}\{Y-\bar y\}$. Then we have
	$K=\operatorname{cone}\{Y-\bar y\}\cap\{-\nabla_yf(\bar x,\bar y)\}^\perp$. Hence, the proof is complete.
	
	(iv-b) 	In case $Y$ is polyhedral convex,  condition (\ref{smrnew}) in condition (i) is equivalent to 
	\[
	-\nabla_{yy}^2 f(\bar x,\bar y)w\in D_*N_Y(\bar y \mid -\nabla_yf(\bar x,\bar y))(w) \implies w=0,
	\]
	which by Definition \ref{strctgd}, is equivalent to 
	\begin{align*}
		\left.	
		\begin{aligned}
			\exists (y^k,\xi^k)\rightarrow(\bar y,-\nabla_yf(\bar x,\bar y)), (w^k,v^k)\rightarrow (w,-\nabla_{yy}^2 f(\bar x,\bar y)w),\\  t_k\downarrow0,\ \mbox{s.t.}\ \xi^k+t_kv^k\in N_Y(y^k+t_kw^k)
		\end{aligned}
		\right\}
		\implies w=0,
	\end{align*}
	which by formula (11) in
	\cite[Lemma 2E.4]{DonRock}, is equivalent to 
	\begin{align*}
		\left.	
		\begin{aligned}
			\exists (y^k,\xi^k)\rightarrow(\bar y,-\nabla_yf(\bar x,\bar y)),  (w^k,v^k)\rightarrow (w,-\nabla_{yy}^2 f(\bar x,\bar y)w),\\ t_k\downarrow0,\ \mbox{s.t.}\ t_kv^k\in N_{K_k}(t_kw^k), K_k=\{d\in T_Y( y^k) \mid d\perp\xi^k\}
		\end{aligned}
		\right\}
		\implies w=0,
	\end{align*}
	which by convexity of $Y$ hence cone $K_k$, is equivalent to 
	\begin{align*}
		\left.	
		\begin{aligned}
			\exists (y^k,\xi^k)\rightarrow(\bar y,-\nabla_yf(\bar x,\bar y)),  (w^k,v^k)\rightarrow (w,-\nabla_{yy}^2 f(\bar x,\bar y)w),\\ t_k\downarrow0,\ \mbox{s.t.}\ v^k\in N_{K_k}(w^k), K_k=\{d\in T_Y( y^k)|d\perp\xi^k\}
		\end{aligned}
		\right\}
		\implies w=0,
	\end{align*}
	which by the inner/outer semi-continuity of normal/tangent cone mapping $N_Y$ and $T_Y$ to polyhedral convex sets, respectively, can be implied by, hence is equivalent to  
	\begin{equation}\label{poly}
		-\nabla_{yy}^2 f(\bar x,\bar y)w\in N_K(w)\implies w=0,
	\end{equation}
	where for simplicity, $K:=K_Y(\bar y,-\nabla_yf(\bar x,\bar y))$ denotes the critical cone to $Y$ for $-\nabla_yf(\bar x,\bar y)$ at $\bar y$. 
	Since $K$ is a closed convex cone, condition (\ref{poly}) is equivalent to that
	\[
	-\nabla_{yy}^2 f(\bar x,\bar y)w\in K^o, ~ w\in K, ~	\langle w,-\nabla_{yy}^2 f(\bar x,\bar y)w\rangle=0\implies w=0,
	\]
	Since $Y$ is a polyhedral convex set, 
	$K=\operatorname{cone}\{Y-\bar y\}\cap\{-\nabla_yf(\bar x,\bar y)\}^\perp$. Then $K^\circ={\rm cl}\, (\{Y-\bar y\}^\circ+{\rm span}\,\{\nabla_yf(\bar x,\bar y)\})$. Hence, the proof is complete.
	\end{proof}


	\begin{remark}
		Obviously, if $\bar y\in S_{\rm FO}(\bar x)$ and $S_{\rm FO}(\bar x)$ has a Lipschitz single-valued localization near $(\bar x,\bar y)$,
		$S_{\rm FO}(x)$ is automatically inner semi-continuous at $(\bar x,\bar y)$. 
		Then one can easily show that the conditions in Theorem \ref{IF} (iv-b) are weaker than 
		the second order sufficient optimality condition
		\[
		\langle w,\nabla_{yy}^2 f(\bar x,\bar y)w\rangle>0~\mbox{for all}~0\neq w\in \{Y-\bar y\} \cap \{\nabla_y f(\bar x,\bar y)\}^\perp.
		\]
	\end{remark}

	
	In principal-agent problems, one always has box constraints $Y=\mathrm{\Pi}_{i=1}^p[a_i,b_i]$. According to the above discussion, in this case, Theorem \ref{IF} has the following form.
	\begin{thm} \label{IF2}
		Suppose $Y=\Pi_{i=1}^p[a_i,b_i]\subseteq\mathbb R^p$. Consider the stationary solution mapping $S_{\rm FO}$, and a pair $(\bar{x}, \bar{y})$ with $\bar{y} \in S_{\rm FO}(\bar{x})$. Suppose that $f$ is continuously differentiable, then 
		\begin{align*}
			&K_Y(\bar y,-\nabla_yf(\bar x,\bar y))\\
			=&\left\{v\in\mathbb R^m\ \left|\ 
			\begin{aligned}
				v_i=0,\quad &\mbox{if}~(\nabla_yf(\bar x,\bar y))_i\neq0,\\
				v_i\in	[0,+\infty),~&\mbox{if}~(\nabla_yf(\bar x,\bar y))_i=0~\mbox{and}~y_i=a_i,\\
				v_i\in	(-\infty,0],~&\mbox{if}~(\nabla_yf(\bar x,\bar y))_i=0~\mbox{and}~y_i=b_i
			\end{aligned}\right.\right\}.
		\end{align*}
		Suppose  one of the following two conditions holds:
		\begin{itemize}
			\item[(a)] the second order sufficient optimality condition holds:
			$$\langle w,\nabla_{yy}^2 f(\bar x,\bar y)w\rangle >0 \quad \forall 0\not =w \in K_Y(\bar y,-\nabla_yf(\bar x,\bar y)),$$
			\item[(b)] the critical cone $K_Y(\bar y,-\nabla_yf(\bar x,\bar y))=\{0\}$.
		\end{itemize}
		Then $S_{\rm FO}$ has a Lipschitz continuous single-valued localization around $(\bar{x},\bar{y})$. 
	\end{thm}
	\begin{proof}
	Since $Y$ is closed and polyhedral convex, and $\{Y-\bar y\}\cap\{\nabla_yf(\bar x,\bar y)\}^\perp=K_Y(\bar y,-\nabla_yf(\bar x,\bar y))$, conditions (a) and (b) both implies condition (iv-a) in Theorem \ref{IF}. Consequently, condition Theorem \ref{IF} (iv-b) also holds.
	\end{proof}

	Now consider the case where $Y=\{y\in\mathbb R^m \mid  g(y) \leq 0\}$ is the solution set of  inequalities, where $ g:\mathbb R^m\rightarrow\mathbb R^p$ is twice continuously differentiable. By \cite[Theorem 6.14]{RW}, for any $y\in Y$, 
	$$\widehat{N}_Y( y)\supseteq \{\nabla g(y)^\top \lambda  |0\leq \lambda \perp g(y)\leq 0\}.$$
	If MSCQ holds, then 
	$${N}_Y( y)\subseteq \{\nabla g(y)^\top \lambda  |0\leq \lambda \perp g(y)\leq 0\}.$$
	It follows that under MSCQ, we have
	\begin{equation}\label{regularity}\widehat{N}_Y( y)={N}_Y( y)= \{\nabla g(y)^\top \lambda  |0\leq \lambda \perp g(y)\leq 0\}.
	\end{equation}
	Therefore, under MSCQ at $ y$, 
	\begin{eqnarray*}
		S_{\rm FO}(x)&=&\{ y| 0=\nabla_y f(x,y)+\nabla g(y)^\top \lambda  |0\leq \lambda \perp g(y)\leq 0\}.
	\end{eqnarray*} 
	Therefore, there  is $V$, a neighborhood of $\bar y$, such that 
	$y\in S_{\rm FO}(x)\cap V$ if and only if there exists a multiplier $\lambda$ such that the  KKT condition holds:
	\begin{align*}
		\nabla_yf(x,y)+\nabla g(y)^T\lambda &=0,\\
		0\leq \lambda \perp g(y)&\leq 0.
	\end{align*}
	By \cite[Theorem 2.2]{HOS}, we have the following localization criterion for the case where the lower level constraints are defined by smooth inequalities. 
	\begin{thm}[Localization Criterion for KKT solution map]\label{IF3}
		Consider the case where $Y=\{y\in\mathbb R^m \mid  g(y) \leq0\}$ and $ g:\mathbb R^m\rightarrow\mathbb R^p$  is twice continuously differentiable. Let $\bar y\in S_{\rm FO}(\bar x)$.	Suppose the following conditions hold at $(\bar x,\bar y)$,
		\begin{itemize}
			\item[(1)]{the MFCQ which is equivalent to NNAMCQ:
				\[
				\nabla g(\bar y)^T\mu=0,~\langle g(\bar y),\mu\rangle=0,~ \mu\geq0\implies \mu=0,
				\] 
				and CRCQ, i.e., there exists some $\varepsilon>0$ such that for any subsets $I\subseteq I_g(\bar y):=\{j=1,\ldots,p \mid g_j(\bar y)=0\}$, the family of gradients $\{\nabla g_j(y)\mid j\in I\}$ has the same rank over the neighborhood $\bar y+\varepsilon\mathbb B$.}

			\item[(2)] for each multiplier $\bar\lambda$   the strong second order sufficient condition holds for $(\bar x,\bar y,\bar\lambda)$:
			\begin{align*}
				\langle y',\nabla^2_{yy}L(\bar x,\bar y,\bar\lambda)y'\rangle>0 \quad \mbox{for all}~ y'\neq0~\mbox{in the subspace}\\
				M=\{y' \mid \nabla g_j(\bar y)y'=0, i=1,\ldots,p, j\in I_{0+}(\bar y,\bar \lambda)\},
			\end{align*}
			where $L(x,y,\lambda):=f(x,y)+\lambda^Tg(y)$ denotes the lower level Lagrange function and $I_{0+}(y,\lambda):=\{j=1,\ldots,q \mid g_j(y)=0,\lambda_i>0\}$. 
		\end{itemize} 
		Then the stationary solution map $S_{\rm FO}$ has a Lipschitz continuous single-valued localization around $\bar x$ for $\bar y$.
	\end{thm}

	\subsection{Sufficient conditions for directional inner semi-continuity}\label{subset3.2}
	Sufficient conditions for (directional) inner semi-continuity of optimal solution map for parametric optimization problems have been studied in literature; see e.g., \cite[Theorem 5.9]{RW} and \cite[Remark 3.2]{DDM}. We now propose following sufficient conditions for the directional inner semi-continuity of $S$.

	Recall that by \cite[Definition 4.5]{BY}, $S(\bar x;u)$ the optimal solution set in direction $u$ at $\bar x$ is a subset of the solution  set $S(\bar x)$ defined by
	\begin{align*}
		S(\bar x;u):=\{y\in S(\bar x) \mid  \exists t_k\downarrow 0, ~u^k\rightarrow u, ~y^k \rightarrow y, ~y^k\in S(\bar x+t_k u^k)\}.
	\end{align*} 
	\begin{prop}\label{suffdinc}
		Let $\bar y\in S(\bar x)$. Suppose the inf-compactness condition holds at $\bar x$, i.e., there exists a compact set $\Omega$ and a real number $\alpha $ such that $
		\{y\in Y \mid f(x,y)\leq \alpha\} \subseteq \Omega$ for all $x $ near $\bar x$. If  $S(\bar x)=\{\bar y\}$ is a singleton, then $S(x)$ is inner semi-continuous at $(\bar x,\bar y)$. If $S(\bar x)\backslash\{\bar y\} \not =\emptyset$ and the direction $u\in\mathbb R^n$ satisfies the directional derivative inequality:
		\begin{align}\label{dirisc}
			u^T \nabla_xf(\bar x,\bar y)<u^T \nabla_xf(\bar x, y) 
			\quad  
			\forall y\in S(\bar x)\backslash\{\bar y\},
		\end{align}
		then the optimal solution set in direction $u$ at $\bar x$ is the singleton $S(\bar x;u)=\{\bar y\}$. Consequently, $S$ is inner semi-continuous at $(\bar x,\bar y)$ in direction $u$.
	\end{prop}
	\begin{proof} First we assume that $S(\bar x)=\{\bar y\}$ and prove that $S(x)$ is inner semi-continuous at $(\bar x,\bar y)$.
	For any sequence $(x^k,y^k)\in\operatorname{gph}S$ with $x^k\rightarrow\bar x$, under the assumed  inf-compactness conditions, $\{y^k\}$ has convergent subsequences. Without loss of generality, let $y^o:=\lim_ky^k$. Since for each $k$, $f(x^k,y^k)\leq f(x^k,\bar y)$, by the continuity of $f$, taking limit as $k\rightarrow\infty$, $f(\bar x,y^o)\leq f(\bar x,\bar y)$. This implies that $y^o\in S(\bar x)=\{\bar y\}$. Hence, $y^k\rightarrow\bar y$ which proves the inner semi-continuous of $S(x)$ at $(\bar x,\bar y)$.
	
	Next suppose that  $\bar y$ is not the only solution in $S(\bar x)$ and the direction $u\in\mathbb R^n$ satisfies the directional derivative inequality (\ref{dirisc}). We first prove $S(\bar x;u)=\{\bar y\}$ by contradiction. To  the contrary, assume that $S(\bar x;u)\not =\{\bar y\}$. Then by the assumed inf-compactness condition, there exist sequences $u^k\rightarrow u, t_k\downarrow0$ and $y^k\in S(\bar x+t_ku^k)$ such that $y^k\rightarrow\tilde y\neq\bar y$. Since $y^k\in S(\bar x+t_ku^k)$ we have $f(\bar x+t_ku^k, y^k)$ and by the continuity of $f$, one can easily obtain $\tilde y\in S(\bar x)$. 
	
	By sensitivity analysis of the value function (\cite[Theorem 4.13]{BS}, $V(x)$ is directionally differentiable and 
	\[
	V'(\bar x;u)=\inf_{y\in S(\bar x)} \nabla_xf(\bar x,y)u \leq \nabla_x f(\bar x,\bar y)u.
	\]
	However, we have
	\begin{align*}
	V'(\bar x;u)&=\lim_k\frac{V(\bar x+t_ku^k)-V(\bar x)}{t_k}\geq\lim_k\frac{f(\bar x+t_ku^k,y^k)-f(\bar x,y^k)}{t_k}\\
	&=\nabla_xf(\bar x,\tilde y)u.
	\end{align*}
	This contradicts the directional derivative inequality (\ref{dirisc}). The proof of $S(\bar x;u)=\{\bar y\}$ is then complete. 
	
	By Definition \ref{incont}, $S$ is inner semi-continuous at $(\bar x,\bar y)$ in direction $u$.
	\end{proof}
	
	\begin{remark}
		In the case where $S(\bar x)=\{\bar y\}$ is a singleton, by Proposition \ref{suffdinc}, optimal solution map $S$ must be inner semi-continuous at $(\bar x,\bar y)$ under the  inf-compactness condition. 
		Generally, if $S(\bar x)$ is a finite set, by the directional derivative inequality (\ref{dirisc}), one can easily find  a direction $u$ along which $S$ is directionally inner semi-continuous. For example, when $f(x,y)$ is a polynomial, the optimal solution $S(\bar x)$ is generically a finite set, and there are many methods to enumerate all its elements; see e.g.,\cite{P1,P2,P3}. Hence, finding such a direction $ u$ is computationally feasible.
	\end{remark}

	The following example illustrate that the inf-compactness condition assumed in Proposition \ref{suffdinc} is useful to guarantee the boundedness of the sequence  $\{y^k\}$ and hence the inner semi-continuity of the solution set. 
	
	\begin{example}
		Consider the unconstrained parametric minimization problem $$\ds{ \min_y f(x,y):=(xy-1)^2(1+y^2)}.$$
		By calculation, $S(x)=\{1/x\}$ if $x\neq0$, and $S(0)$=\{0\}.
		Obviously, $S(x)$ is single-valued but not continuous at $(0,0)$. It does not contradict Proposition \ref{suffdinc} since the inf-compactness condition does not hold at $x=0$.
	\end{example}

	
	The following proposition gives  some sufficient conditions for the inner semi-continuity of the following solution map \begin{equation}\label{eqn13} S(x):=\arg\min_y\{f(x,A(y)) \mid y\in Y\}.
	\end{equation}
	\begin{prop}\label{insc}
		Let $\bar y\in S(\bar x)$. Suppose the inf-compactness condition holds at  $\bar x$,  the map $A(y)$ is a homeomorphism on $Y$, 
		$f(\bar x,z)$ is strictly convex in variable $z$ and the set  $A(Y)$ is convex. Then the solution map $S(\bar x)$ defined as in (\ref{eqn13}) is equal to the  singleton $\{\bar y\}$ and $S$ is inner semi-continuous at   $(\bar x,\bar y)$.
	\end{prop}
	\begin{proof}
	Obviously, since $A(y)$ is a homeomorphism we have 
	\begin{eqnarray*}
		S(\bar x)
		&=&A^{-1}(\arg\min_z\{ f(\bar x,z) \mid z \in A(Y)\}).
	\end{eqnarray*} 
	By the strict convexity of $f(\bar x,z)$ in variable $z$, the convexity of set $A(Y)$ we conclude that the solution set $\arg\min_z\{ f(\bar x,z) \mid z \in A(Y)\}$ is a singleton which implies that the solution $S(\bar x)$ is a singleton. Combining with  the inf-compactness condition, we complete the proof by using Proposition \ref{suffdinc}.
	\end{proof}

	\begin{example}
		Consider the  nonconvex mathematical program with parameter $x\in \mathbb R$:
		\begin{align*}
			\min_y~~xy^3+y^{12},\quad\mathrm{s.t.}~ y\in(-\infty,+\infty).
		\end{align*}
		The inf-compactness condition obviously holds.
		It is clear that  $z:=A(y)=y^3$ is a homeomorphism from $\mathbb R$ to $\mathbb R$, the function $f(x,z)=xz+z^4$  is strictly convex in variable $z$, and $A(Y)=\mathbb R$ is a convex set. By Proposition \ref{insc}, $S(x)$ is inner semicontinuous everywhere. Indeed, we can check that for any $x$, $S(x)=-\sqrt[9]{\frac{x}{4}}$ which is single-valued and continuous.
	\end{example}
	
	\subsection{Applications to the principal-agent problems}\label{PA_application}

	We now demonstrate that the  results in the previous subsections \ref{subset3.1}-\ref{subset3.2} can be applied to  the principal-agent problem. In the literature of the principal-agent problems (see e.g., Grossman and Hart \cite{GH83}),     the probability of output $s$ given action $y$  often satisfies the so-called the spanning condition, i.e., 
\[
P(s, y)= p(y) Q(s) + \left( 1- p(y) \right) H(s), \quad \mathrm{for\ all\ outputs}\ s \in S=\{s_1, \dots, s_n\},
\]
where $p(y)\in [0,1]$ for all $y\in [a, b]$, and $Q,H$ are non-identical probability mass functions.  Also it is common to assume that  $p(y)$ is strictly increasing on $[a,b]$.

In this setting, adopting the conventional assumptions $c(y)=y$ and $u(x_j)=x_j(j=1,\ldots,n)$ (assuming the agent is  risk-neutral  or applying a transformation  $\tilde{x}_j=u(x_j)$), and using the formulas in \eqref{pa0}, the lower-level objective function is given by 
\[
f(x,y)=-p(y) \sum_{j=1}^n x_j \left[ Q(s_j) - H(s_j) \right]
+
y
- \sum_{j=1}^n x_j  H(s_j), 
\]
while the upper-level objective function is
\[
F(x, y) 
= -p(y) \sum_{j=1}^n v(\pi_j - x_j ) \left[ Q(s_j) - H(s_j) \right]
- \sum_{j=1}^n v(\pi_j - x_j ) H(s_j).
\]
If $p(y)$ is strictly increasing on $[a,b]$, one has $$\nabla_yf(x,y)=0\implies\sum_{j=1}^n x_j [ Q(s_j) - H(s_j) ]>0,$$ for any $y\in S_{\rm FO}(x)$. Hence, if further $p(y)$ is concave, the classical first order approach is proven to be valid, see, e.g., Proposition 12 in \cite{JK15}. 
Now, we demonstrate that we can relax the concavity condition on $p(y)$ to the more challenging non-concave scenario by using  Theorems~\ref{genthm} and~\ref{IF2} and Proposition~\ref{suffdinc}. 
To this end, note that when $p(y)$ is strictly increasing, the upper-level objective function $F(x, \cdot)$ is monotone for any fixed $x$. Therefore, if $(\bar{x}, \bar{y})$ is an optimal solution of \eqref{pa0}, then $\bar{y}$ not only belongs to $S(\bar{x})$, but also must be either the leftmost or rightmost point of $S(\bar{x})$. 
The latter fact allows us to show that there exists a direction $u$ satisfying the directional derivative condition \eqref{dirisc}. Indeed, using the expression 
\[
\nabla_x f(x, y)=-p(y)[Q(s_1)-H(s_1),\dots,Q(s_n)-H(s_n)]^T-[H(s_1),\dots, H(s_n)]^T,
\] 
for any $y \in S(\bar{x}) \setminus \{\bar y\}$, we have 
\[
\nabla_xf(\bar x,\bar y)u- \nabla_xf(\bar x, y)u 
=-[ p(\bar y) -p(y)  ]  [Q(s_1)-H(s_1),\dots,Q(s_n)-H(s_n)]^T u.
\]
Since $Q$ and $H$ are not identical, one can choose $u$ such that the inner product $[Q(s_1)-H(s_1),\dots, Q(s_n)-H(s_n)]^T u$ is either positive or negative (for example, choose 
$u=\pm [Q(s_1)-H(s_1),\dots, Q(s_n)-H(s_n)]$. Furthermore, since $\bar{y}$ is either the leftmost or rightmost point of $S(\bar{x})$ and $p(y)$ is strictly increasing, $p(\bar y) - p(y)$ remains either positive or negative for all $y \in S(\bar{x}) \setminus {\bar y}$. Hence, there exists a direction $u$ satisfying \eqref{dirisc}.  
Consequently, Proposition~\ref{suffdinc} implies that $S$ is inner semi-continuous at $(\bar x, \bar y)$ in the direction $u$.
Since $Y=[a,b]\subseteq\mathbb R$ is a closed interval, 
if further $\nabla_{yy}^2 f(\bar x, \bar y) \neq 0$ (i.e., $p''(\bar y) \neq 0$) or $\nabla_yf(\bar x,\bar y)\neq0$ (i.e., $p'(\bar y)(Q(s)-H(s))^T\bar x+1 \neq 0$), then $S_{\rm FO}$ has a Lipschitz single-valued localization near $(\bar x, \bar y)$, as shown in Theorem~\ref{IF2}. Therefore, Theorem~\ref{genthm} establishes the local equivalent characterization of the lower-level program constraint over a directional neighborhood of the reference point. It is worth noting that we allow the lower-level objective function to be nonconvex and permit non-unique solutions throughout the process. 

\section{First order necessary optimality conditions for (\ref{BP})}\label{firstcondition}

In Section \ref{Equivreform}, we propose a directional equivalent characterization of the lower level program constraint. Now we apply it to (\ref{BP}) and obtain a  directional single level reformulation for (\ref{BP}), through which we can establish first order necessary optimality conditions. 

First, we define the following  directional  SCOP reformulation for (\ref{BP}).
\begin{defn}\label{dirMPECrefor}
	Let $(\bar x,\bar y)$ be a feasible point of {\rm (\ref{BP})}. We say that {\rm (\ref{BP})} has an SCOP reformulation at $(\bar x,\bar y)$ along direction $u\in\mathbb{R}^n$, if there exists positive scalars $\varepsilon,\delta$  such that the two constraints
	\[
	y\in S(x)~~\mbox{and}~~(y,-\nabla_yf(x,y))\in{\rm gph}\, \widehat N_{Y}
	\]
	are equivalent over the directional neighborhood $(\bar x,\bar y)+\mathcal{V}_{\varepsilon,\delta}(u)\times\varepsilon\mathbb B$.
\end{defn}	

{Based on this directional SCOP reformulation, (\ref{BP}) can be written into a single level optimization program, denoted as (SCOP$_u$). However, the KKT conditions of (SCOP$_u$) involves the normal cone to a directional neighborhood, $N_{\bar x+\mathcal V_{\varepsilon,\delta}(u)}(\bar x)$, which is hard to be computed. To deal with this difficulty, we adopt the following directional local optimality theory.}

Recently, Ouyang et al. \cite[Definition 3.1]{OYZ2025} introduced the directional local optimality for set constrained optimization problems. For the sake of completeness. We apply the definition to problem (SCOP).
\begin{defn}
	A point $(\bar x,\bar y)$ is said to be a local optimal solution of {\rm (SCOP)} in direction $(u,v)\in\mathbb R^n\times\mathbb R^m$, if there exist positive scalars $\varepsilon,\delta$ such that
	\begin{align*}
		F(x,y)\geq F(\bar x,\bar y)
		\quad \forall (x,y)\in \mathcal F\cap \left((\bar x,\bar y)+\mathcal V_{\varepsilon,\delta}(u,v)\right),
	\end{align*}
	where $\mathcal F\subseteq\mathbb R^n\times\mathbb R^m$ denotes the feasible set of {\rm (SCOP)}, i.e.,
	\begin{align*}
		\mathcal F:=\left \{(x,y) \mid G(x,y)\leq0, (y,-\nabla_yf(x,y))\in{\rm gph}\,\widehat N_{Y}\right \}.
	\end{align*}
\end{defn}

By  the definition of a directional neighborhood in Definition \ref{dn}, we have  that for any direction $v\in\mathbb R^m$, there exists $\delta>0$ sufficiently small such that
\[
(\bar x,\bar y)+{\cal V}_{\bar\varepsilon,\bar\delta}(u)\times\bar\varepsilon\mathbb B\supseteq(\bar x,\bar y)+{\cal V}_{\bar\varepsilon,\delta}(u,v).
\] Consequently, we have the following result.

\begin{prop}\label{dirmp}
	If $(\bar x,\bar y)$ solves $\mathrm{( SCOP_{u})}$, then for any direction $v\in\mathbb R^m$, there exist positive scalars $\bar\varepsilon,\bar\delta$ such that $(\bar x,\bar y)$ also solves program
	\begin{equation*}\tag{SCOP\textsubscript{(u,v)}}
		\begin{aligned}
			\min_{x,y}\  &F(x,y)\\
			{\rm s.t.}\  & G(x,y)\leq0,\
			(y,-\nabla_yf(x,y))\in{\rm gph}\,{\widehat N}_{Y},\\
			&(x,y)\in(\bar x,\bar y)+{\cal V}_{\bar\varepsilon,\bar\delta}(u,v).
		\end{aligned} 
	\end{equation*}
	Consequently, $(\bar x,\bar y)$ is also a local optimal solution of {\rm (SCOP)} in direction $(u,v)$ for any $v\in\mathbb R^m$.
\end{prop}  

By Proposition \ref{dirmp}, to derive optimality conditions for (SCOP$_{u}$), it suffices to study optimality conditions for (SCOP$_{u,v}$) for certain $v$. Since the emergence of directional variational analysis and its application to optimality theory, there have been many works on optimality conditions for various mathematical programs with the constraint $(x,y)\in(\bar x,\bar y)+{\cal V}_{\bar\varepsilon,\bar\delta}(u,v)$; see e.g., \cite{BY,OYZ2025} and the references therein. Particularly, \cite{OYZ2025} studied the case of set constrained optimization problems and developed first order optimality conditions for any directional local optimal solution. In the following theorem, based on the relation in Proposition \ref{dirmp} and the optimality condition in \cite[Proposition 4.1]{OYZ2025}, we obtain KKT type necessary optimality conditions for (SCOP$_{u,v}$). 

Define the linearized cone of (SCOP) at $(\bar{x},\bar{y})$ by 
\begin{eqnarray*}
	L(\bar x,\bar y)&&:=\left\{(u,v)
	\left|
	\begin{array}{ll}
		\nabla G_i(\bar x,\bar y)(u,v)\leq0,~ \forall i\in I_G,\\
		\nabla(y,-\nabla_yf)(\bar x,\bar y)(u,v)\in T_{\mbox{gph}\,{\widehat N}_{Y}}(\bar y,-\nabla_yf(\bar x,\bar y))
	\end{array}
	\right.	\right\},
\end{eqnarray*} 
where $I_G:=\{i=1,\ldots,q \mid G_i(\bar x,\bar y)=0\}$. Under the metric subregularity, by Proposition \ref{dirAbadie}, $L(\bar x,\bar y)$ is included in the tangent cone to the feasible region of problem (SCOP). Hence, for any $(u,v)\in L(\bar x,\bar y)$ one always has $\nabla F(\bar x,\bar y)(u,v)\geq0$. Consequently, the existence of $(u,v)\in L(\bar x,\bar y)$ and $\nabla F(\bar x,\bar y)(u,v)=0$ at a local optimal solution $(\bar x,\bar y)$ means the existence of a non-ascent direction.
\begin{prop}\label{gopt}
	Let $(\bar x,\bar y)$ be a local optimal solution of {\rm (SCOP)} in direction $(u,v)$. Suppose that $(u,v)\in L(\bar x,\bar y)$ and $\nabla F(\bar x,\bar y)(u,v)=0$.  Suppose the metric subregularity holds for the system $$G(x,y)\leq0,\
	(y,-\nabla_yf(x,y))\in{\rm gph}\,{\widehat N}_{Y}$$ at $(\bar x,\bar y)$ in direction  $(u,v)$. Then there exists  $(\bar\mu,\bar\nu)\in\mathbb R^{2m}$ and $\bar\beta\in\mathbb R^{q}_+$ such that the M-stationary condition in direction $(u,v)$ holds, i.e.,
	\begin{eqnarray}\label{dirKKT}
		\nonumber
		&&	0=\nabla F(\bar x,\bar y)+\nabla( y, -\nabla_yf)(\bar x,\bar y)^T
		(\bar\mu,\bar\nu)+\nabla G(\bar x,\bar y)^T {\bar\beta},\\
		&& (\bar\mu,\bar\nu)\in N_{{\rm gph}{\widehat N}_{Y}}(\bar y,-\nabla_yf(\bar x,\bar y);v,-\nabla(\nabla_yf)(\bar x,\bar y)(u,v)),\\
		\nonumber
		&&\bar\beta\perp G(\bar x,\bar y),\ \bar\beta\perp \nabla G(\bar x,\bar y)(u,v).
	\end{eqnarray}
\end{prop}
\begin{proof}
Since $(u,v)\in L(\bar x,\bar y)$ and the directional metric subregularity holds in direction $(u,v)$, by Proposition \ref{dirAbadie} $(u,v)$ is a feasible direction of the constraint system. Since $\nabla F(\bar x,\bar y)(u,v)=0$, conditions in \cite[Proposition 4.1]{OYZ2025} are all satisfied. The optimality conditions in (\ref{dirKKT}) then follow from 	 Definition \ref{dnorm}, the Cartesian product rule for directional normal cones \cite[Proposition 3.3]{YZ18}, 
and the formula (\ref{convNormal}),
\begin{align*}
	N_{\mathbb R^q_-}(G(\bar x,\bar y);\nabla G(\bar x,\bar y)(u,v))&=N_{\mathbb R^q_-}(G(\bar x,\bar y))\cap\{\nabla G(\bar x,\bar y)(u,v)\}^\perp.
\end{align*}
\end{proof}


Now we are ready to give a directional M-stationary condition as a necessary optimality condition for (\ref{BP}).		
\begin{thm}\label{opt}
	Let $(\bar x,\bar y)$ be a local minimizer of {\rm (\ref{BP})}.
	Suppose that  {\rm (\ref{BP})} has an SCOP reformulation at $(\bar x,\bar y)$ along direction $u$. Suppose there exists a direction $v\in\mathbb R^m$ such that $(u,v)\in L(\bar x,\bar y)$ and $\nabla F(\bar x,\bar y)(u,v)=0$.  Suppose the metric subregularity holds for the system $$G(x,y)\leq0,\
	(y,-\nabla_yf(x,y))\in{\rm gph}\,{\widehat N}_{Y},$$ at $(\bar x,\bar y)$ in direction  $(u,v)$.
	Then the directional KKT condition holds. That is, there exists  $(\bar\mu,\bar\nu)\in\mathbb R^{2m}$ with $\bar\beta\in\mathbb R^{q}_+$ such that\
	(\ref{dirKKT}) holds.
\end{thm}
\begin{proof}
The proof  follows by combining Definition \ref{dirMPECrefor}, Proposition \ref{dirmp} and Proposition \ref{gopt}.
\end{proof}

\begin{remark}		The directional metric subregularity, though weak, due to its abstract form, it is generally difficult to be verified by its definition. In  Proposition \ref{qp}, we list some verifiable sufficient conditions for directional metric subregularity, including the NNAMCQ, FOSCMS and ``affine+polyhedral'' condition. Applying these sufficient conditions to the system 
	$$( G(x,y), y,-\nabla_yf(x,y)) \in\mathbb R^q_- \times {\rm gph}\,{\widehat N}_{Y},$$	
	we have that NNAMCQ holds at $(\bar x,\bar y)$ if and only if there exists no nonzero vector $((\mu,\nu),\beta)\in\mathbb R^{2m}\times \mathbb{R}^q_+$   such that 
	\begin{eqnarray*}
		&&	0=\nabla( y,-\nabla_yf)(\bar x,\bar y)^T(\mu,\nu)+\nabla G(\bar x,\bar y)^T \beta,\\
		&&(\mu,\nu)\in N_{{\rm gph}{\widehat N}_{Y}}(\bar y,-\nabla_yf(\bar x,\bar y)),\ \beta\perp G(\bar x,\bar y), 
	\end{eqnarray*} and  FOSCMS holds in direction $(u,v)$ satisfying $$\nabla( y,-\nabla_yf)(\bar x,\bar y)(u,v) \in T_{{{\rm gph}{\widehat N}_{Y}}}(\bar y,-\nabla_yf(\bar x,\bar y)) $$ if and only if there 
	exists no nonzero vector $((\mu,\nu),\beta)\in\mathbb R^{2m}\times \mathbb{R}^q_+$   such that 
	\begin{eqnarray*}
		&&	0=\nabla( y,-\nabla_yf)(\bar x,\bar y)^T(\mu,\nu)+\nabla G(\bar x,\bar y)^T \beta, \beta\perp G(\bar x,\bar y), {\beta\perp \nabla G(\bar x,\bar y)(u,v)},\\
		&&(\mu,\nu)\in N_{{\rm gph}{\widehat N}_{Y}}(\bar y,-\nabla_yf(\bar x,\bar y);v,-\nabla(\nabla_yf)(\bar x,\bar y)(u,v)),
	\end{eqnarray*}
	and ``affine+polyhedral'' condition means that $Y$ is polyhedral and $G$ is affine.
	
\end{remark}

In the above constraint qualification and the necessary optimality condition, one needs to calculate $N_{{\rm gph}\widehat N_{Y}}(\bar y,-\nabla_yf(\bar x,\bar y))$ or its directional version $$ N_{{\rm gph}{\widehat N}_{Y}}(\bar y,-\nabla_yf(\bar x,\bar y);v,-\nabla(\nabla_yf)(\bar x,\bar y)(u,v)).$$ 		
Note that in case where $Y=[a,b]$ is an interval, since ${\rm gph}\widehat N_{Y}$ is the union of segments $(\{a\}\times\mathbb R_-)\cup([a,b]\times\{0\})\cup(\{b\}\times\mathbb R_+)$, we have
\begin{eqnarray*}
	&&N_{{\rm gph}N_{Y}}(\bar y,-\nabla_yf(\bar x,\bar y))\\
	=&&\left\{
	\begin{array}{ll}
		\mathbb R\times\{0\}	,\ &\mbox{if}~\bar y=a, \nabla_yf(\bar x,\bar y)>0,\\
		(\mathbb R_-\times\mathbb R_+)\cup(\mathbb R\times\{0\})\cup(\{0\}\times\mathbb R)	,\ &\mbox{if}~\bar y=a, \nabla_yf(\bar x,\bar y)=0,\\
		\{0\}\times\mathbb R	,\ &\mbox{if}~\bar y\in(a,b), \nabla_yf(\bar x,\bar y)=0,\\
		(\mathbb R_+\times\mathbb R_-)\cup(\mathbb R\times\{0\})\cup(\{0\}\times\mathbb R)	,\ &\mbox{if}~\bar y=b, \nabla_yf(\bar x,\bar y)=0,\\
		\mathbb R\times\{0\}	,\ &\mbox{if}~\bar y=b, \nabla_yf(\bar x,\bar y)<0.
	\end{array}\right.
\end{eqnarray*}
In the case where $Y=\mathrm{\Pi}_{i=1}^p[a_i,b_i]\subseteq\mathbb R^p$ is a box, one can apply the Cartesian product rule of the  normal cone and use the above formula. 

In the rest of this section, we will give our necessary optimality condition in details for two special and common cases.


\subsection{The case where $\bar y$ lies in the interior of $Y$}
In case $\bar y\in \operatorname{int}Y$, $\widehat{N}_{Y}(y)\equiv\{0\}$ near $\bar y$. Then $N_{{\rm gph}\widehat N_{Y}}(\bar y,-\nabla_yf(\bar x,\bar y))=\{0\}\times\mathbb R^m$. Hence Theorem \ref{opt} takes the following simple form in this case.
\begin{thm}\label{opt1}
	Let $(\bar x,\bar y)$ be a local minimizer of {\rm (\ref{BP})} 
	with $\bar y\in \operatorname{int}Y$. 
	Suppose the inf-compactness condition holds at $\bar x$ and $\nabla_{yy}^2f(\bar x,\bar y)$ is invertible. Suppose either $S(\bar x)=\{\bar y\}$ (e.g., the solution map is defined by (\ref{eqn13}) and  the assumptions in Proposition \ref{insc} hold) or $S(\bar x)$ is not equal to  $\{\bar y\}$ but there exists some direction $u\in\mathbb R^n$ satisfying  the directional derivative inequality (\ref{dirisc}) and  there exists a direction $v$ such that $\nabla \nabla_y f(\bar x,\bar y)(u,v)=0, \nabla G(\bar x,\bar y)(u,v) \leq 0$ and $\nabla F(\bar x,\bar y)(u,v)=0$. 
	Also, suppose that the system $G(x,y)\leq0, \nabla_y f(x,y)=0$ is metric subregularity at $(\bar x,\bar y)$ in the direction $(u,v)$. 
	Then, the directional M-stationary condition holds; that is, there exists a pair  $(\bar\nu,\bar\beta)\in\mathbb R^{m+q}$ with $\bar\beta\in\mathbb R^{q}_+$ such that
	\begin{eqnarray*}
		\nonumber
		&&	0= \nabla F(\bar x,\bar y)-\nabla(\nabla_{y}f)(\bar x,\bar y)^T\bar\nu
		+\nabla G(\bar x,\bar y)^T {\bar\beta},\\
		&& 
		\bar\beta\perp G(\bar x,\bar y),\ \bar\beta\perp \nabla G(\bar x,\bar y)(u,v).
	\end{eqnarray*}
\end{thm}				

In the following example, we show that the first order approach fails in the classical sense but can hold along certain direction. This illustrates the advantage of directional first order approach for bilevel program without lower level convexity.		

\begin{example}\label{modified Mirrlees}
	Recall that the feasible region in Mirrlees’ example exhibits a jump at the point $(1, y_0)$, where $0<y_0=0.957<1$ and satisfies the condition $(1-y_0)e^{4y_0}/(1+y_0)=1$. 
	Here we present a modification of Mirrlees’ example featuring a two dimensional upper variable.
	\begin{equation}\label{modifiedMirrlees}
		\begin{aligned}
			\min \quad &F(x_1, x_2,  y)=(x_1-0.5)^2+(x_1+x_2)(1+y)-(1-y)e^{4y} + (y-y_0)^3\\
			\mathrm { s.t. }\quad
			& y \text { minimizes }  f(x_1,x_2, y)=-(x_1+x_2) e^{-(y+1)^2}-e^{-(y-1)^2} . 
		\end{aligned}
	\end{equation}
	Similar to Mirrlees’ example, the global optimal solutions of the lower level problem form a disconnected hypersurface that exhibits a jump along the line $x_1+x_2=1$ on the smooth hypersurface 
	\begin{equation*}
		S_{\mathrm{FO}}:=\left\{
		(x_1,x_2,y)\in\mathbb{R}^3  \mid  \frac{1-y}{1+y}e^{4y}=x_1+x_2
		\right\},
	\end{equation*}
	which is determined by the first order condition $\nabla_y f(x,y)=0$. 
	Moreover, one can verify that all bilevel feasible points in a small neighborhood $U_0$ of $(\bar{x},\bar{y}):=(0.5,0.5,y_0)^T$ satisfy $y\geq y_0$.
	
	We first claim that $(\bar{x},\bar{y})$ is a local minimizer of \eqref{modifiedMirrlees}. Indeed, since $F(\bar{x},\bar{y})=0$, it suffices to show that $F(x,y)\geq 0$ in a small neighborhood of $(\bar{x},\bar{y})$ within the bilevel feasible region. 
	To this end, note that for any bilevel feasible point $(x,y)\in U_0$, the first order condition $\nabla_y f(x,y)=0$ implies that $F(x,y)=(x_1-0.5)^2+ (y-y_0)^3\geq0$ since $y\geq y_0$ for all bilevel feasible points in $U_0$. 
	
	Secondly, for any neighborhood of $(\bar{x},\bar{y})$ on the hypersurface $S_{\mathrm{FO}}$ defined by the first order condition, there exists a point $(\bar{x},y)$ such that $F(\bar{x},y)=(y-y_0)^3<0$. Hence, the point $(\bar{x},\bar{y})$ is not a local minimizer of $F(x,y)$ on $S_{\mathrm{FO}}$.
	This demonstrates that the first order approach fails at $(\bar{x},\bar{y})$. 
	
	Thirdly, according to Mirrlees' example, $S(\bar x)=\{-0.957, 0.957\}$.  Since we have $\nabla_x f(\bar x,0.957)u<\nabla_xf(\bar x,-0.957)u$, the directional derivative inequality (\ref{dirisc}) holds. By Proposition \ref{suffdinc}, $S$ is inner semi-continuous at $(\bar x,\bar y)$ along direction $u$. Moreover $\nabla_{yy}^2 f(\bar x,\bar y)\not =0$ and hence by Theorem \ref{IF}, $S_{\rm FO}(x)$ has a single-valued localization near $ (\bar x,\bar y)$. Then combining Theorem \ref{genthm}, and Proposition \ref{dirmp}, $(\bar{x},\bar{y})$ is a local solution of (SCOP)  which takes the form 
	$$ \min F(x,y)~ \mbox{ s.t. }  \nabla_y f(x,y)=0$$	
	in direction $(u,v)$ for any $v\in\mathbb R^m$. 
	
	Finally, there is $v=-\nabla_{yy}^2f(\bar x,\bar y)^{-1}\nabla_{xy}^2f(\bar x,\bar y)u$ such that $$\nabla \nabla_y f(\bar x,\bar y)(u,v)=0~\mbox{and}~\nabla F(\bar x,\bar y)(u,v)=0.$$ That is $(u,v)$ is a nonzero critical direction.
	Since $\nabla_{yy}^2 f(\bar x,\bar y)\not =0$, the Linear Independence constraint qualification holds for the system $\nabla_y f(x,y)=0$ at $(\bar x,\bar y)$.
	Applying Theorem \ref{opt1}, the directional KKT condition holds at the optimal solution $(\bar x,\bar y)$. Indeed, there exists a vector $\bar\nu =\frac{1}{2}e^{(y_0+1)^2}\approx 23.1 \in\mathbb R$ such that
	\begin{eqnarray*}
		0= \nabla F(\bar x,\bar y)-\nabla(\nabla_{y}f(x,y))(\bar x,\bar y)^T\bar\nu.
	\end{eqnarray*}			
\end{example}
\subsection{The case where the lower level constraints are defined by smooth inequalities}		

In this subsection we derive the optimality conditions for the case where the lower level constraints are defined by twice continuously differentiable  inequalities.


\begin{thm}\label{KKT1}
	Let $(\bar x,\bar y)$ be a local minimizer of {\rm (\ref{BP})} where $Y=\{y\in\mathbb R^m \mid  g(y) \leq 0\}$ and $g:\mathbb R^m\rightarrow\mathbb R^p$ is twice continuously differentiable.
	Suppose the LICQ holds  for the system  $g(y)\leq0$ at $\bar y$,
	and  the strong second order sufficient condition holds for $(\bar x,\bar y,\bar\lambda)$ where $\bar \lambda$ is the unique multiplier.
	Suppose the inf-compactness condition holds at $\bar x$ and there exists some direction $u\in\mathbb R^n$ satisfying the directional derivative inequality (\ref{dirisc}). Suppose there exists $v\in\mathbb R^m$ and $e^*\in\mathbb R^p$ such that $\nabla F(\bar x,\bar y)(u,v)=0$ and 
	\begin{align}\label{lcone}
			\begin{aligned}
				& 0=\nabla ( \nabla_y f+\nabla g^T \bar \lambda) (\bar x,\bar y) (u,v)+\nabla g(\bar y)^Te^*\\
				& \nabla G_i(\bar x,\bar y)(u,v)\leq0,\ i\in I_G(\bar x,\bar y),\
				\nabla g_j(\bar y){}v\leq0,\ j\in I_g(\bar y),
			\end{aligned}
		\end{align} 
		where $e^*_j=0$ if either $g_j(\bar y)<0$ or $g_j(\bar y)=\bar\lambda_j=0, \nabla g_j(\bar y)v<0$, and $e^*_j\geq0$ if $g_j(\bar y)=\bar\lambda_j=\nabla g_j(\bar y)
		{^T}v=0$, $j=1,\ldots,p$.
	Suppose FOSCMS holds at $(\bar x,\bar y)$	in direction $(u,v)$. That is, there is no nonzero $(\nu,\beta)\in\mathbb R^{m}\times \mathbb R^q_+$ and $b^* \in \mathbb R^p$ such that   
	\begin{eqnarray*}
		\nonumber
		&&	0={ -\nabla(\nabla_yf+\nabla g^T\bar \lambda)(\bar x,\bar y)^T { \nu}+(0, \nabla g(\bar y)^Tb^*)}+\nabla G(\bar x,\bar y)^T {\beta},\\ &&\beta\perp G(\bar x,\bar y),\
		\beta\perp \nabla G(\bar x,\bar y)(u,v),\ {\nabla g_j(\bar y)\nu\leq0,\ j\in I_g(\bar y),}
	\end{eqnarray*}
	where  $ b^*_j=0$ if either $g_j(\bar y)<0$ or $g_j(\bar y)= \bar\lambda_j=0, \nabla g_j (\bar y) \nu <0$, and $ b^*_j\geq0$ if $g_j(\bar y)= \bar\lambda_j=\nabla g_j (\bar y) v=e^*_j=0, \nabla g_j (\bar y)^T\nu <0$ for $j=1,\ldots,p$.
	Then  the M-stationary condition in direction $(u,v)$ holds, i.e., there exist $(\bar\nu,\bar\beta)\in\mathbb R^{m}\times \mathbb R^q_+$ and $b^* \in \mathbb R^p$ such that
	\begin{align*}
		&	0=\nabla F(\bar x,\bar y)-\nabla(\nabla_yf+\nabla g^T\bar \lambda)(\bar x,\bar y)^T {\bar \nu}+(0, \nabla g(\bar y)^Tb^*)+\nabla G(\bar x,\bar y)^T {\bar\beta},\\ &\bar\beta\perp G(\bar x,\bar y),\
		\bar\beta\perp \nabla G(\bar x,\bar y)(u,v),\\
		& b_j^*=0 \mbox{ if  either } g_j(\bar y)<0 \mbox{ or } g_j(\bar y)=\bar \lambda_j=0, \nabla g_j (\bar y)v<0, e^*_j=0;\\
		& \nabla g_j (\bar y)^T \bar \nu =0 \mbox{ if } \text{either } g_j(\bar y)=\bar \lambda_j=\nabla g_j (\bar y)v=0, e^*_j>0,\\
		&\qquad\qquad\qquad~~~ \mbox{ or } g_j(\bar y)=\nabla g_j (\bar y)v=0, \bar \lambda_j>0;\\
		&\mbox{either } b^*_j>0, \nabla g_j (\bar y)^T \bar \nu <0 \mbox{ or }  b_j^* \nabla g_j (\bar y)^T \bar \nu =0\\ 
		&\qquad\qquad\qquad\qquad\qquad\qquad\qquad\qquad \mbox{ if } g_j(\bar y)= \bar \lambda_j =\nabla g_j (\bar y)v=\bar \lambda_j=0.
	\end{align*}

		\end{thm}
		\begin{proof}
		By Theorem \ref{IF3}, Proposition \ref{suffdinc} and Theorem \ref{genthm}, the constraints $y\in S(x)$ and $(y,-\nabla_yf(x,y))\in{\rm gph}\,\widehat{N}_{Y}$  are equivalent over the directional   neighborhood of  $(\bar x,\bar y)$ in direction $(u,v)$.\ It follows that $(\bar x,\bar y)$ is also  a local minimizer of problem (SCOP) in direction $(u,v)$.
		
		By (\ref{regularity}), 
		\begin{eqnarray*}{\rm gph}{\widehat N}_{Y}  &=&\left \{(y, \nabla g(y)^\top \lambda ) |   \lambda \in N_{\mathbb{R}_-^p}( g(y))\right \}.
		\end{eqnarray*}

		Since LICQ holds for the system $g(y)\leq 0$ at $\bar y$,   by   \cite[Theorem 4.2]{BGYZZ}, we have 
		\begin{eqnarray*}
			&&T_{{\mathrm gph}\, {\widehat N}_{Y}}(\bar y,\nabla g(\bar y)^T\bar\lambda)\\
			&=&\left\{\left(
			\begin{array}{c}
				v\\
				(\nabla ({ \nabla g})(\bar y)v)^T\bar\lambda+\nabla g(\bar y)^Te^*
			\end{array}\right)\left|\left(
			\begin{array}{c}
				\nabla g(\bar y)v\\
				e^*
			\end{array}
			\right)\in T_{{\mathrm gph}\, N_{\mathbb R^p_-}}(g(\bar y),\bar\lambda)
			\right.\right\},
		\end{eqnarray*}
		where for each $j=1,\dots, p$,
		\begin{eqnarray*}	
			&& T_{{\mathrm gph}\, { N}_{\R_- }}( g_j(\bar y),\bar\lambda_j)
			= \left\{
			\begin{array}{ll}
				\mathbb R \times \{0\},~~& \mathrm{if}~g_j(\bar y)<0, \bar \lambda_j=0, \\
				\mathbb R_-\times \{0\} \cup \{0\} \times \mathbb R_+ ,~~&\mathrm{if}~g_j(\bar y)=\bar\lambda_j=0, \\
				\{0\} \times \mathbb R ,~~&\mathrm{if}~g_j(\bar y)=0, \bar\lambda_j>0.
			\end{array}
			\right. 
		\end{eqnarray*}
		Condition $(u,v)$ being a critical direction of SCOP means that $ (u,v)\in L(\bar x,\bar y)$ and 
		$\nabla F(\bar x,\bar y)(u,v)=0$. Since the linearized cone $L(\bar x,\bar y)$ is defined by the set of directions $(u,v)$ satisfying that 
		\begin{align*}
		&\nabla G_i(\bar x,\bar y)(u,v)\leq 0{(\forall i\in I_G(\bar x,\bar y))},\\ &(v, -\nabla (\nabla_y f)(\bar x,\bar y)(u,v))\in T_{{\mathrm gph}\,{\widehat N}_{Y}}(\bar y,\nabla g(\bar y)^T\bar\lambda),
		\end{align*}
		we obtain (\ref{lcone}).
		We now apply Theorem \ref{opt}. Under FOSCMS at $(\bar x,\bar y)$ in direction $(u,v)$, condition (\ref{dirKKT})---the KKT condition  in direction $(u,v)$ holds, that is,  there exists $((\bar\mu,\bar\nu),\bar\beta)\in\mathbb R^{2m}\times \mathbb R^q_+$ such that 
		\begin{eqnarray*}
			\nonumber
			&&	0=\nabla F(\bar x,\bar y)+(0, \bar \mu) -\nabla (\nabla_y f)(\bar x,\bar y)\bar \nu+\nabla G(\bar x,\bar y)^T {\bar\beta},\\
			&&\ \bar\beta\perp G(\bar x,\bar y),
			\bar\beta\perp \nabla G(\bar x,\bar y)( u, v),\\
			&&(\bar\mu,\bar\nu)\in N_{{\rm gph}{\widehat N}_{Y}}(\bar y,-\nabla_yf(\bar x,\bar y);v,-\nabla(\nabla_yf)(\bar x,\bar y)(u,v)).
		\end{eqnarray*}
		
		Let $\bar \zeta =\nabla g(\bar y)^\top \bar \lambda=-\nabla_yf(\bar x,\bar y)$. By   \cite[Theorem 4.2]{BGYZZ}
		\begin{eqnarray*}
			& & N_{{\rm gph}{\widehat N}_{Y}}(\bar y,\bar \zeta;v,-\nabla(\nabla_yf)(\bar x,\bar y)(u,v)))\\
			\qquad \qquad &=&
			\left\{ (\mu, \nu) \left|
			\begin{array}{l}
				\mu =-\nabla(\nabla g(\bar y)^T \bar \lambda) \nu
				+\nabla g(\bar y)^Tb^*  \\
				(b^*, \nabla g(\bar y) \nu) \in N_{{\rm gph} N_{\R_-^p}}((g(\bar y), \bar {\lambda});\nabla g(\bar y)v, e^*)
			\end{array}\right.
			\right \}.
		\end{eqnarray*} Since for any real numbers $a,b,c,d$ with $a\leq 0, b\geq 0, c\leq 0$, we have 
		\begin{eqnarray*}
			&& 
			N_{{\rm gph}N_{\R_-}}((a,b);(c,d))\\
			&& =
			\begin{cases}
				\{0\} \times \R & a<0, b=0, c\in \R, d=0, \\
				\{0\} \times \R & a=0, b=0, c<0, d=0,\\
				(\{0\} \times \R)\cup (\R\times \{0\})\cup (\R_+\times \R_-) & a=b=c=d=0,\\
				\R\times \{0\} & a=0, b=0, c=0, d>0,\\
				\R\times \{0\} & a=0, b>0, c=0, d\in \R,
			\end{cases}
		\end{eqnarray*}
		{taking into account the Cartesian product rule of directional normal cone (see \cite[Proposition 3.2]{YZ18})}	the rest of the proof follows. 	
		\end{proof}

		In the following theorem we extend the classical result of \cite[Theorem 4.1]{YY} to allow the lower level program to be nonconvex.
		\begin{thm}\label{KKT2}
			Let $(\bar x,\bar y)$ be a local minimizer of {\rm (\ref{BP})} where $Y=\{y\in\mathbb R^m \mid  g(y) \leq 0\}$ and $g:\mathbb R^m\rightarrow\mathbb R^p$ is twice continuously differentiable.
			Suppose the lower level solution  $S(\bar x)=\{\bar y\}$ is unique, the inf-compactness condition holds at $\bar x$ and LICQ holds  for the system  $g(y)\leq0$ at $\bar y$,
			and  the strong second order sufficient condition holds for $(\bar x,\bar y,\bar\lambda)$ where $\bar \lambda$ is the unique multiplier.
			Suppose the inf-compactness condition holds at $\bar x$.  
			If  the metric subregularity holds for the system 
			\begin{equation}\label{eqn4.4} G(x,y)\leq0,\
				(y,-\nabla_yf(x,y))\in{\rm gph}\,{\widehat N}_{Y}
			\end{equation} at $(\bar x,\bar y)$,  then   the M-stationary condition holds, i.e., there exist $(\bar\nu,\bar\beta)\in\mathbb R^{m}\times \mathbb R^q_+$ and $ b^*\in \R^p$ such that 
			\begin{eqnarray*}
				\nonumber
				&&	0=\nabla F(\bar x,\bar y)-\nabla(\nabla_yf+\nabla g^T\bar \lambda)(\bar  x,\bar y)^T  \bar \nu+(0,  \nabla g(\bar y)^T b^*) +\nabla G(\bar x,\bar y)^T {\bar\beta},\ \\
				&& \bar\beta\perp G(\bar x,\bar y),\\
				&& b_j^*=0 \mbox{ if } g_j(\bar y)<0, \bar \lambda_j =0;\
				\nabla g_j (\bar y)^T \bar \nu =0 \mbox{ if } g_j(\bar y)=0, \bar \lambda_j >0;\\
				&& \mbox{ either } b^*_j>0, \nabla g_j (\bar y)^T \bar \nu <0 \mbox{ or }  b_j^* \nabla g_j (\bar y)^T \bar \nu =0 \mbox{ if } g_j(\bar y)= \bar \lambda_j =0.
			\end{eqnarray*}
			
		\end{thm}
		\begin{proof} By Theorems \ref{genthm}, \ref{IF3} and Proposition \ref{suffdinc}, (\ref{BP}) allows SCOP reformulation at $(\bar x,\bar y)$ in direction $(u,v)=(0,0)$. That is, 
		$(\bar x,\bar y)$ is a local minimizer of SCOP.   By \cite[Theorem 3.2]{YY}, under the metric subregularity for the system (\ref{eqn4.4}),  
		there exists $((\bar\mu,\bar\nu),\bar\beta)\in\mathbb R^{2m}\times \mathbb R^q_+$ such that
		\begin{eqnarray*}
			\nonumber
			&&	0=\nabla F(\bar x,\bar y)+(0, \bar \mu) -\nabla (\nabla_y f)(\bar x,\bar y)\bar \nu
			+\nabla G(\bar x,\bar y)^T {\bar\beta},\\
			&& (\bar\mu,\bar\nu)\in N_{{\rm gph}{\widehat N}_{Y}}(\bar y,-\nabla_yf(\bar x,\bar y)),\bar\beta\perp G(\bar x,\bar y).
		\end{eqnarray*}
		Since 
		\begin{eqnarray*}
			N_{{\rm gph}{\widehat N}_{Y}}(\bar y,-\nabla_yf(\bar x,\bar y))&=&
			\left\{ (\mu, \nu) \left|
			\begin{array}{l}
				\mu =-\nabla(\nabla g(\bar y)^T \bar \lambda) \nu
				+\nabla g(\bar y)^Tb^*  \\
				(b^*, \nabla g(\bar y) \nu) \in N_{{\rm gph} N_{\R_-^p}}((g(\bar y), \bar {\lambda}))
			\end{array}\right.	\right \}		
		\end{eqnarray*}
		and \[
		N_{{\rm gph}N_{\R_-}}(g_j(\bar y), \bar {\lambda}_j)=
		\begin{cases}
			\{0\}
			\times \R & g_j(\bar y)<0, \bar {\lambda}_j=0, \\
			(\{ 0\} \times \R) \cup (\R \times \{0\})\cup (\R_+\times \R_-) & g_j(\bar y)=0,\bar {\lambda}_j=0,\\
			\R\times \{0\} & g_j(\bar y)=0, \bar {\lambda}_j>0,
		\end{cases}
		\] the desired results follow.
		
		\end{proof}

				\section{Conclusion}			
				
				It is known that for (\ref{BP}) with nonconvex lower level program, the classical first order approach usually fails and the value function approach, as well as the combined approach, leads to nonsmooth optimization programs with degenerate constraint system. In this paper, we propose a directional first order approach for deriving necessary optimality conditions for (\ref{BP}). When the lower level program is convex, this approach recovers the classical first order approach by taking zero as the direction. And for (\ref{BP}) with nonconvex lower level program, our analysis demonstrated that the first order approach works over some directional neighborhood in the presented nonconvex setting. Hence, we obtain an equivalent single level reformulation for general (\ref{BP}) without the value function.  Applying the directional optimality theory,  some directional constraint qualifications can be applied to establish directional KKT conditions for (\ref{BP}). On the other hand, the application of our approach is limited by the assumption of existence of a non-ascent direction.


				

\end{document}